 \documentclass[a4paper,11pt]{article}
\usepackage{pdfsync}

 \usepackage[plainpages=false]{hyperref}
 \usepackage{amsfonts,amsmath,amssymb,amsthm,enumitem}
 \usepackage{latexsym,lscape,rawfonts,mathrsfs}
\usepackage [autostyle, english = american]{csquotes}

 \usepackage[dvips]{color}
 \usepackage{multicol}

 \usepackage[all]{xy}
 \usepackage{eufrak}
 \usepackage{makeidx}
 \usepackage{graphicx,psfrag}
 \usepackage{pstool}

 \usepackage{array,tabularx}

 \usepackage{setspace}

\usepackage[titletoc]{appendix}

\usepackage{txfonts}

 \newcommand{\ba}{\begin{align*}}
 \newcommand{\ea}{\end{align*}}

 \makeatletter
 \def\ExtendSymbol#1#2#3#4#5{\ext@arrow 0099{\arrowfill@#1#2#3}{#4}{#5}}
 
 \makeatother

 \makeatletter
 \def\ExtendSymbol#1#2#3#4#5{\ext@arrow 0099{\arrowfill@#1#2#3}{#4}{#5}}
 
 \makeatother

 \definecolor{hao}{rgb}{1,0.5,0}
 \definecolor{miao}{cmyk}{0.5,0,0.2,0.2}
 \definecolor{qiao}{gray}{0.96}

\newtheorem{thm}{Theorem}[section]
\newtheorem{cor}[thm]{Corollary}

\newtheorem{lem}[thm]{Lemma}
\newtheorem{conj}[thm]{Conjecture}

\theoremstyle{definition}
\newtheorem{defn}[thm]{Definition}

\theoremstyle{remark}
\newtheorem{rem}[thm]{Remark}

\numberwithin{equation}{section}

 \title{Ricci flow on asymptotically Euclidean manifolds}
 \author{Yu Li}

\begin{document}
\maketitle

\begin{abstract}
In this paper, we prove that if an asymptotically Euclidean manifold with nonnegative scalar curvature has long time existence of Ricci flow, the ADM mass is nonnegative. In addition, we give an independent proof of positive mass theorem in dimension three.
\end{abstract}

\tableofcontents

\section{Introduction}
A smooth orientable Riemannian manifold $(M^n,g)$ ($n \ge 3$) is called an \emph{Asymptotically Euclidean} (AE) manifold if for some compact $K \subset M^n$, $M^n \backslash K$ consists of a finite number of components $E_1,\ldots, E_k$ such that for each $E_i$
 there exists a $C^{\infty}$ diffeomorphism $\Phi_i:E_i \to \mathbb{R}^n \backslash B(0,A_i)$ such that under this identification,
\begin{equation} \label{E100}
g_{ij} =\delta_{ij}+O(r^{-\sigma_i}), \quad \partial^{|k|}g_{ij} =O(r^{-\sigma_i-k})
\end{equation} 
for any partial derivative of order $k$ as $r \to \infty$, where $r$ is the Euclidean distance function. We call the positive number $\sigma_i$ the order of end $E_i$.

The ADM mass \cite{ADM61} from general relativity of an AE manifold $(M,g)$ is defined as 
$$
m(g)=\lim_{r \to \infty} \int_{S_r}(\partial_ig_{ij}-\partial_j g_{ii})\,dA^j,
$$
where $dA^j=\partial_j \lrcorner dV_{g_E}$ and $g_E$ is the canonical Euclidean metric on $\mathbb R^n$.

The definition of mass involves a choice of asymptotic coordinates. But it follows from Bartnik's result \cite{B86} that if the order $\sigma>(n-2)/2$ and the scalar curvature is integrable, then the mass is finite and independent of AE coordinates. In other words, $m(g)$ depends only on the metric $g$.

The general positive mass conjecture is the following, see \cite[Theorem $10.1$]{LP87}.
\begin{conj}[Positive Mass Conjecture]
Let $(M^n,g)$ be an AE manifold of dimension $n \ge 3$ with the order $\sigma >(n-2)/2$, and nonnegative integrable scalar curvature. Then $m(g) \ge 0$ with equality if and only if $(M,g)=(\mathbb R^n,g_E)$. 
\end{conj}

In dimension three, the positive mass conjecture was first proved by Schoen and Yau \cite{SY79} in 1979 by constructing a stable minimal surface and considering its stability inequality. In addition, Schoen and Yau showed that their method could be extended to the case when the dimension was less than eight \cite{R81,SY792}. In 1981, Witten \cite{W81} proved the positive mass conjecture for spin manifolds of any dimension. In 2001, Huisken and Ilmanen \cite{GT01} proved the stronger Riemannian Penrose inequality in dimension three by using the inverse mean curvature flow. In 2015, Hein and LeBrun gave a proof of the positive mass conjecture for K\"ahler AE manifolds, see \cite{HL15}. To the author's knowledge, there is no proof of the positive mass conjecture in general dimension.

A natural question arises, can we prove the positive mass conjecture by using other geometric flows? Since Ricci flow is one of the most powerful geometric flows by which Perelman have completely solved \text{Thurston's} geometrization conjecture, see \cite{P03,P031,P032}, it is of interest to know how Ricci flow interacts with AE manifolds and the ADM mass. 

Recall that Ricci flow is a geometric flow such that a family of metrics $g(t)$ on a smooth manifold $M$ are evolved under the PDE
\begin{align}\label{E001}
\partial_t g(t)=-2Rc(g(t)).
\end{align}
We will focus on the case when $(M,g(0))$ is an AE manifold.

It has been proved by Dai and Ma in \cite{DM07} that Ricci flow preserves the ALE condition, nonnegative integrable scalar curvature and the ADM mass. Hence, it is important to understand the change of mass at possible singular times and infinity if long time existence of Ricci flow is assumed.

One of the main theorems in this paper shows that if we have long time existence of Ricci flow, an AE manifold will converge to the Euclidean space in some strong sense. The proof is partially motivated by considering possible steady solitons on ALE manifolds, see Appendix. The convergence at time infinity will indicate that the mass is nonnegative along the flow.

We assume throughout this paper that the scalar curvature $R$ is nonnegative and integrable, the manifold has only one end $E$ \footnote{In fact, all the arguments below apply to the multi-end case with slight modifications.}and the order of the end $\sigma$ is greater than $(n-2)/2$. Moreover, we fix a positive smooth function $r(x)$ on $M$ such that $r(x)=|\Phi(x)|$ when $x \in E$, where $\Phi$ is the diffeomorphism in the definition of AE manifolds. We also identify $x \in E$ with $\Phi(x) \in \mathbb R^n$ without explicitly mentioning $\Phi$.

Moreover, we assume that the order $\sigma \le n-2$ since if an AE manifold is of order greater than $n-2$, then it is also of order $n-2$.

\begin{thm}\label{T101}
Let $(M^n,g)$ be an AE manifold with above assumptions. If there exists a solution $g(t) \, (0 \le t < \infty)$ of the Ricci flow with $g(0)=g$, then the mass $m(g) \ge 0$ with the equality if and only if $(M^n,g)=(\mathbb R^n,g_{E})$.
\end{thm}

Under Ricci flow, it is possible that the metric becomes singular at some finite time. In dimension three, we can continue Ricci flow by performing surgeries. We prove that the mass and other related conditions are preserved under Ricci flow with surgery. Moreover, if we choose surgery parameter function $\delta(t)$ small enough, there are only finitely many surgeries. The finiteness of surgeries is proved by carefully examining the change of Perelman's $\mu$-functional over surgery times.
By choosing one appropriate Ricci flow with surgery, we have the long time existence of Ricci flow after the last surgery time and Theorem \ref{T101} applies.

\begin{thm}\label{T102}
When $n=3$, the mass $m(g) \ge 0$ with the equality if and only if $(M^3,g)=(\mathbb R^3,g_{E})$.
\end{thm}

For the remainder of the paper, $C$ may vary from line to line. Moreover $\Delta=\Delta_{g(t)}$, $\nabla =\nabla_{g(t)}$ and $dV=dV_{g(t)}$ unless otherwise specified.

{\bf Acknowledgements}: I would like to express my gratitude to my advisor, Professor Bing Wang. He brought this problem to my attention and steered me in the right direction. I am also grateful to Professor Xiuxiong Chen and Professor G\'abor Sz\'ekelyhidi for their helpful discussions.

\section{Mass under Ricci flow}
We prove in this section that Ricci flow preserves the AE condition and the mass is unchanged under Ricci flow. Different from the argument of Dai and Ma in \cite{DM07}, we fix an AE coordinate system along the flow. The main tool we use is the following maximum principle on the noncompact manifold with evolving metrics, see \cite[Theorem $12.14$]{CCGGIIKLLN10}.

\begin{thm} \label{T103}
Suppose that $g(t)$, $t\in [0,T]$, is a complete solution to the Ricci flow on a noncompact manifold $M$ with $|\text{Rm}(g(t))| \le k_0$ for some $k_0>0$. Let 
$$
Lu=u_t-\Delta u-\langle X(t),\nabla u \rangle-G(u,t),
$$
where $X(t)$ is a smooth family of bounded vector fields and the function $G: \mathbb R \times [0,T] \to \mathbb R$ is locally Lipschitz in the $\mathbb R$ factor and continuous in the $[0,T]$ factor. Suppose that $u$ is a smooth function such that
$$
Lu \le 0 \quad \text{and} \quad |u(x,t)|\le \exp \left( b(d_{g(t)}(O,x)+1) \right)
$$
for some constant $b$. For any $c \in \mathbb R$, let $U(t)$ be the solution to the corresponding ordinary differential equation:
\begin{align}
\frac{dU}{dt}=G(U,t), \quad U(0)=c. \notag
\end{align}
If $u(x,0)\le c$ for all $x \in M$, then we have
$$
u(x,t) \le U(t)
$$
for all $x \in M$ and $t \in [0,T]$ as long as the ODE exists.
\end{thm}

\begin{thm}\label{R01}
Let $(M,g(t))$ be a Ricci flow solution with bounded curvature on $M\times [0,T]$ and $(M,g(0))$ is an AE manifold of order $\sigma >0$, then
\begin{enumerate}[label=(\roman*)]
\item  AE condition is preserved with the same AE coordinates and order.
\item If $\sigma>(n-2)/2$ and $R$ is integrable, the mass is unchanged.
\end{enumerate}  
\end{thm}

\begin{proof}
$(i)$: Since $(M,g(0))$ is an AE manifold, there exists an end $E$ and a $C^{\infty}$ diffeomorphism $\Phi: E \to \mathbb R^n \backslash B(0,A)$ such that under this coordinate system
\begin{equation} \label{ER01}
g_{ij} =\delta_{ij}+O(r^{-\sigma}), \quad \partial^{|k|}g_{ij} =O(r^{-\sigma-k})
\end{equation} 
for all $k=1,2,\cdots$. \\
From this condition, it is easy to conclude that $|\nabla^k\text{Rm}(0)|=O(r^{-\sigma-k-2})$.

Since the Riemannian curvature is uniformly bounded on $[0,T]$, there exists an $S>0$ such that $|\text{Rm}|\le S$ on $M\times [0,T]$. Now we consider the evolution equation of $|\text{Rm}|^2$ \cite[$(2.57)$, $(6.1)$]{CLN06}
$$
\partial_t|\text{Rm}|^2 \le \Delta |\text{Rm}|^2+16|\text{Rm}|^3 \le \Delta |\text{Rm}|^2+16S|\text{Rm}|^2.
$$
Let $u=|\text{Rm}|^2e^{-16St}$, then $\partial_tu \le \Delta u$ on $M\times [0,T]$. \\
Next we prove that $u$ has the same spatial decaying condition as $u(0)$, see also \cite{DM07}.\\
Let $h(x)=r^{4+2\sigma}$ on $M$. We set $w=hu$ and it satisfies
$$
(\partial_t-\Delta)w \le Bw-2\nabla\log h\nabla w
$$
on $M\times [0,T]$, where $B=\frac{2|\nabla h|^2-h\Delta h}{h^2}$.  

We first show that under $|\text{Rm}| \le S$, $B$ is uniformly bounded on $M\times [0,T]$.

Since $|\text{Rm}| \le S$, the metrics $g(t)$ are uniformly comparable to $g(0)$. That is,
\begin{align}\label{E201f}
C^{-1}g(0) \le g(t) \le Cg(0)
\end{align}
on $M \times [0,T]$.

Now we have the following evolution equations for $|\nabla h|^2=|\nabla_{g(t)}h|_{g(t)}^2$ and $\Delta h=\Delta_{g(t)}h$,
\begin{align}
\partial_t |\nabla h|^2&=2Rc(\nabla h,\nabla h),\label{E201b} \\ 
\partial_t (\Delta h)&=2\langle Rc,\nabla^2h \rangle.\label{E201c} 
\end{align}
The proof of \eqref{E201b} is straighforward and the proof of \eqref{E201c} can be found in \cite[Lemma $2.30$]{CLN06}.

Therefore, from the curvature bound and \eqref{E201f}
\begin{align}\label{E201d}
|\partial_t |\nabla h|^2| &\le C|\nabla h|^2, \\ 
|\partial_t  (\Delta h)| & \le C |\nabla^2 h| \le C|\nabla^2_{g(0)}h|_{g(0)},
\end{align}
and by integration
\begin{align}\label{E201e}
|\nabla_{g(t)}h|_{g(t)}^2 &\le C|\nabla_{g(0)}h|_{g(0)}^2, \\ 
|\Delta_{g(t)}h| & \le C |\nabla^2_{g(0)}h|_{g(0)}.
\end{align}

To estimate $|\nabla_{g(0)}h|_{g(0)}^2$ and $|\nabla^2_{g(0)}h|_{g(0)}$, we use the given coordinate system of $g(0)$ at infinity. From the definition of $h$ and direct computations, it is easy to show that
\begin{align}\label{E201g}
|\nabla_{g(0)}h|_{g(0)}^2 &\le Cr^{6+4\sigma}, \\ 
|\nabla^2_{g(0)}h|_{g(0)} &\le Cr^{2+2\sigma}.
\end{align}

Therefore we have
\begin{align}\label{E201h}
|B|=\left|\frac{2|\nabla h|^2-h\Delta h}{h^2}\right| \le C\left|\frac{|\nabla_{g(0)}h|_{g(0)}^2}{h^2}\right|+C\left|\frac{|\nabla^2_{g(0)}h|_{g(0)}}{h}\right| \le Cr^{-2} \le C
\end{align}
where the last inequality is true since $r$ has a positive minimum.

From Theorem \ref{T103}, we conclude that $|w| \le C$ and hence $|\text{Rm}| \le Cr^{-2-\sigma}$ on $M\times [0,T]$.\\
\emph{Claim}: 
\begin{equation}\label{E2002}
|\nabla^k\text{Rm}| \le Cr^{-2-k-\sigma}. 
\end{equation}
\emph{Proof of the claim}:
We assume that the claim holds for all $0 \le l <k$.
Let $h_k=r^{4+2\sigma+2k}$ and $w_k=h_k|\nabla^k\text{Rm}|^2$, then from the evolution equation of $|\nabla^k \text{Rm}|^2$ \cite[$(6.24)$]{CLN06}
\begin{align}\label{E201a}
\partial_t|\nabla^k\text{Rm}|^2&=\Delta|\nabla^k\text{Rm}|^2-2|\nabla^{k+1}\text{Rm}|^2+\sum_{l=0}^k\nabla^l\text{Rm}*\nabla^{k-l}\text{Rm}*\nabla^k\text{Rm} \notag \\
& \le \Delta|\nabla^k\text{Rm}|^2+C\sum_{l=0}^k|\nabla^l\text{Rm}||\nabla^{k-l}\text{Rm}||\nabla^k\text{Rm}|
\end{align}
we have  
\begin{equation}\label{E2011}
(\partial_t-\Delta)w_k \le B_kw_k-2\nabla\log h_k\nabla w_k+C\sum_{l=0}^kh_k|\nabla^l\text{Rm}||\nabla^{k-l}\text{Rm}||\nabla^k\text{Rm}|
\end{equation}
where $B_k=\frac{2|\nabla h_k|^2-h_k\Delta h_k}{h_k^2}$ is uniformly bounded as before.
Moreover, by induction we have
$$
h_k|\nabla^l\text{Rm}||\nabla^{k-l}\text{Rm}||\nabla^k\text{Rm}| =h_k|\text{Rm}||\nabla^k\text{Rm}|^2 \le Cw_k
$$
for $l=0$ or $l=k$ and 
$$
h_k|\nabla^l\text{Rm}||\nabla^{k-l}\text{Rm}||\nabla^k\text{Rm}| \le h_kr^{-4-2\sigma-k}|\nabla^k\text{Rm}|=r^k|\nabla^k\text{Rm}| \le Cw_k^{1/2}
$$
for $0<l<k$.\\
From \eqref{E2011} we have
$$
(\partial_t-\Delta)w_k \le -2\nabla\log h_k\nabla w_k+C(w_k+w_k^{1/2}).
$$
From Theorem \ref{T103}, we conclude that $w_k$ is uniformly bounded on $M \times [0,T]$ since the the solution of the following ODE
\begin{align}
\frac{d\phi}{dt}&=C(\phi+\phi^{1/2}), \notag \\ 
\phi(0)&=c
\end{align}
is bounded on $[0,T]$. Therefore $|\nabla^k\text{Rm}| \le Cr^{-2-k-\sigma}$.

For any vector field $U$ on $M$, we have
\begin{align}
&|\log g(x,t)(U,U)-\log g(x,0)(U,U)| \notag \\
=&\left|\int_0^{t}\frac{-2Rc(x,s)(U,U)}{g(x,s)(U,U)} \,ds\right | \le C\int_0^{t}|\text{Rm}| \, ds \le Cr^{-\sigma-2}.
\end{align}

Therefore
\begin{equation}
g(t)(U,U)=g(0)(U,U)(1+O(r^{-2-\sigma})),
\end{equation}
and in particular,
\begin{align}\label{RE03}
g_{ii}(t)&=g_{ii}(0)(1+O(r^{-2-\sigma})) \notag \\
&=(1+O(r^{-\sigma}))(1+O(r^{-2-\sigma})) \notag \\
&=1+O(r^{-\sigma})
\end{align}
By the polarization identity and \eqref{RE03}, we conclude that $g_{ij}(t)=O(r^{-\sigma})$ when $i \ne j$.

Now from the evolution equation of the Christoffel symbol \cite[$(2.25)$]{CLN06}
$$
\partial_t\Gamma_{ij}^k=-g^{kl}(\nabla_iR_{jl}+\nabla_jR_{il}-\nabla_lR_{ij})
$$
and \eqref{E2002}, we conclude that $\Gamma_{ij}^k=O(r^{-\sigma-3})$ and hence $\partial_iR_{jk}=O(r^{-\sigma-3})$ from the relation $\nabla_iR_{jk}=\partial_iR_{jk}-\Gamma_{ij}^lR_{lk}-\Gamma_{ik}^lR_{jl}$.

Since $\partial_t( \partial_ig_{jk})=-2\partial_iR_{jk}$, it follows that $\partial_ig_{jk}(t)=O(r^{-\sigma-1})$. Now by induction, $\partial^{|k|}g_{ij} =O(r^{-\sigma-k})$ for all $k$ and hence $(E,g_{ij}(t))$ is an AE coordinate system with the same order $\sigma$.

$(ii)$: From the definition of the mass
$$
m(g(t))=\lim_{r \to \infty} \int_{S_r}(\partial_ig_{ij}(t)-\partial_j g_{ii}(t))\,dA^j.
$$
Since we have a common coordinate system at infinity,
\begin{align}
m'(g(t))&=\lim_{r \to \infty} \int_{S_r}(\partial_ig'_{ij}(t)-\partial_j g'_{ii}(t))\,dA^j \notag \\ \notag
&=\lim_{r \to \infty} -2\int_{S_r}(\partial_iR_{ij}(t)-\partial_j R_{ii}(t))\,dA^j\\ \notag
&=\lim_{r \to \infty} -2\int_{S_r}(\nabla_i R_{ij}(t)-\nabla_j R(t))\,dA^j \\ \notag
&=\lim_{r \to \infty} \int_{S_r}\nabla_j R(t)\,dA^j. 
\end{align}

Now from \cite[Lemma $11$]{MS12},
$$
\lim_{r \to \infty} \int_{S_r}|\nabla R(t)|\,d\sigma=0
$$
for $t>0$, so $m'(g(t))=0$ for $t>0$. 

On the other hand, it is easy to show that $m(g(t))$ is continuous at $0$, see  \cite[Corollary $12$]{MS12}, hence the mass is unchanged.
\end{proof}

\begin{rem}
The proof of Theorem \ref{R01} actually shows that if $g_{ij}(0)-\delta_{ij}\in C^k_{-\sigma}$, then $g_{ij}(t)-\delta_{ij}\in C^{k-2}_{-\sigma}$ for any integer $k \ge 4$ and $t>0$. In addition, using the argument in \cite{DM07} we can prove that if  $g_{ij}(0)-\delta_{ij}\in C^2_{-\sigma}$, then $g_{ij}(t)-\delta_{ij}\in C^{1,\alpha}_{-\sigma}$ for $t>0$. The definition of the weighted space can be found in Section $5$.
\end{rem}

Let $(M,g(t)), 0\le t \le T$ be a Ricci flow solution with bounded curvature on $M\times [0,T]$ such that $(M,g(0))$ is an AE manifold. By our assumption, the scalar curvature $R(x,0)\ge 0$. From the evolution equation of $R$, that is, $\partial_t R=\Delta R+2|Rc|^2 \ge \Delta R$ and Theorem \ref{T103}, $R(x,t) \ge 0$ on $M\times [0,T]$.

Now from the strong maximum principle under Ricci flow \cite[Lemma $6.57$]{CLN06}, either $R(x,t) >0$ for $(x,t)\in M \times (0,T]$ or $R(x,t)=0$ on $M\times [0,T]$.

In the first case, we redefine the Ricci flow $g_1(t)=g(t+\epsilon_1)$ where $\epsilon_1 \in (0,T)$ is fixed such that the corresponding scalar curvature $R_1(x,0) >0$ for all $x \in M$.

In the second case, the evolution equation of $R$ implies that $Rc(0)=0$, that is, $(M,g(0))$ is Ricci-flat. Now we have
\begin{thm}\label{R02}
If $(M,g)$ is a Ricci-flat AE manifold, then $(M,g)$ is isometric to $(\mathbb R^n,g_E)$.
\end{thm}

We fix a point $p$ on $M$ and let $d(x)=d_g(x,p)$ be the distance function to $p$. We first prove the following two lemmas.

\begin{lem}\label{LR01}
\begin{equation}\label{R02a}
\lim_{r\to +\infty}\frac{r(x)}{d(x)}=1 
\end{equation}
where $r$=$r(x)$.
\end{lem}

\begin{proof}
From the definition of AE manifolds, there exists a large positive number $r_0$ such that
\begin{align} \label{R02a}
(1+Cr^{-\sigma})^{-1}g_E(x) \le g(x) \le (1+Cr^{-\sigma})g_E(x).
\end{align}
for all $r(x) \ge r_0$.

Given $r_1 \ge r_0$ and large $r(x)$, let $\left\{\gamma (t), \, t\in [0,d(x)]\right \}$ be the minimizing geodesic from $p$ to $x$. Then there exists an $r_x \in [0,d(x)]$ such that $r(\gamma (r_x))=r_1$ and $r(\gamma(t)) \ge r_1$ for $t \in [r_x,d(x)]$. We assume that 
$r_x \in [C_1^{-1}r_1,C_1r_1]$ where $C_1$ depends on $r_1$.

Now we estimate the distance between $\gamma(r_x)$ and $x$ under $g_E$. We have
\begin{align} \label{R02d}
r(x)-r_1 \le \int_{r_x}^{d(x)}|\gamma'(t)|_{g_E} \, dt \le (1+Cr_1^{-\sigma})\left(\int_{r_x}^{d(x)}|\gamma'(t)|_{g} \, dt \right) \le (1+Cr_1^{-\sigma})d(x)
\end{align}
where we have used \eqref{R02a} to estimate $|\gamma'(t)|_{g_E}$.
Then we obtain from \eqref{R02d} that
\begin{align} \label{R02d1}
r(x)\le  (1+Cr_1^{-\sigma})d(x)+r_1.
\end{align}

On the other hand, let $\left\{\gamma_1 (t), \, t\in [0, a]\right \}$ be the minimizing geodesic from $\gamma(r_x)$ to $x$ under $g_E$. Similarly we have
\begin{align} \label{R02e}
d(x)-r_x \le \int_{0}^{a}|\gamma_1'(t)|_{g} \, dt \le  (1+Cr_1^{-\sigma})\left(\int_{0}^{a}|\gamma_1'(t)|_{g_E} \, dt \right) \le  (1+Cr_1^{-\sigma})(r(x)+r_1)
\end{align}
and hence
\begin{align} \label{R02e1}
d(x) \le  (1+Cr_1^{-\sigma})(r(x)+r_1)+r_x \le  (1+Cr_1^{-\sigma})r(x)+(1+Cr_1^{-\sigma}+C_1)r_1.
\end{align}

Combining \eqref{R02d1} and \eqref{R02e1}, we have
\begin{align} \label{R02g3}
(1+Cr_1^{-\sigma})^{-1} \le \liminf_{r \to +\infty}\frac{r(x)}{d(x)} \le \limsup_{r \to +\infty}\frac{r(x)}{d(x)}\le 1+Cr_1^{-\sigma}.
\end{align}
Since $r_1$ can be chosen as large as we want,
\begin{equation}\label{R02h}
\lim_{r\to +\infty}\frac{r(x)}{d(x)}=1 
\end{equation}
 and the proof of the lemma is complete.
\end{proof}

\begin{lem}\label{LR02}
\begin{equation}\label{ER01}
\lim_{r\to +\infty}\frac{\text{Vol}_gB(p,d(x))}{w_{n}r^n(x)}=1, 
\end{equation}
where $w_n$ is the volume of the unit ball in $\mathbb R^n$.
\end{lem}
\begin{proof}
For the AE manifold, there exists an $r_0>0$ sufficiently large such that
\begin{align} \label{ER02}
(1+Cr^{-\sigma})^{-1}g_E(x) \le g(x) \le (1+Cr^{-\sigma})g_E(x)
\end{align}
and hence
\begin{align} \label{ER03}
(1+Cr^{-\sigma})^{-1}\text{Vol}_{g_E}(x) \le \text{Vol}_{g}(x) \le (1+Cr^{-\sigma})\text{Vol}_{g_E}(x)
\end{align}
for any $r(x) \ge r_0$.

For any $r(x) \ge r_0$, from Lemma \ref{LR01} there exists a function $\epsilon (r)>0$ with $\epsilon(r) \to 0$ as $r \to +\infty$ such that
\begin{align} \label{ER04}
e^{-\epsilon(r)}\le \frac{r(x)}{d(x)} \le e^{\epsilon(r)}.
\end{align}

Now we fix an $r_1 \ge r_0$. Then for any $r(x) > r_1$, we have
\begin{align} \label{ER05}
w_n\left((e^{-\epsilon(r)}r)^n-r_1^n\right)&=\text{Vol}_{g_E}\left(B(0,e^{-\epsilon(r)}r)\backslash B(0,r_1)\right) \\ \notag
&\le (1+Cr_1^{-\sigma})\text{Vol}_{g}\left (B(0,e^{-\epsilon(r)}r)\backslash B(0,r_1)\right)\\ \notag
&\le (1+Cr_1^{-\sigma})\text{Vol}_g(B(p,d))
\end{align}
where the last inequality is true since by \eqref{ER04}, $B(0,e^{-\epsilon(r)}r)\backslash B(0,r_1) \subset B(p,d)$. Hence
\begin{align} \label{ER06}
w_n(e^{-\epsilon(r)}r)^n \le (1+Cr_1^{-\sigma})\text{Vol}_g(B(p,d))+w_nr_1^n.
\end{align}

On the other hand,
\begin{align} \label{ER07}
\text{Vol}_{g}\left(B(p,d)\backslash B(p,e^{\epsilon(r_1)}r_1)\right) & \le \text{Vol}_{g}\left(B(0,e^{\epsilon(r)}r)\backslash B(0,r_1)\right)\\ \notag
&\le (1+Cr_1^{-\sigma})\text{Vol}_{g_E}\left(B(0,e^{\epsilon(r)}r)\backslash B(0,r_1)\right)\\ \notag
&= (1+Cr_1^{-\sigma})w_n\left((e^{\epsilon(r)}r)^n-r_1^n\right)
\end{align}
and hence
\begin{align} \label{ER08}
\text{Vol}_{g}(B(p,d)) \le (1+Cr_1^{-\sigma})w_n(e^{\epsilon(r)}r)^n+\text{Vol}_{g}(B(p,e^{\epsilon(r_1)}r_1)).
\end{align}

Combining \eqref{ER06} and \eqref{ER08}, we have
\begin{align} \label{ER09}
(1+Cr_1^{-\sigma})^{-1}\le \liminf_{r\to +\infty}\frac{\text{Vol}_gB(p,d(x))}{w_{n}r^n} \le \limsup_{r\to +\infty}\frac{\text{Vol}_gB(p,d(x))}{w_{n}r^n} \le 1+Cr_1^{-\sigma}
\end{align}

By taking $r_1$ to $+\infty$, we conclude that
\begin{equation}\label{ER10}
\lim_{r\to +\infty}\frac{\text{Vol}_gB(p,d(x))}{w_{n}r^n}=1, 
\end{equation}
\end{proof}

\emph{Proof of Theorem \ref{R02}}: From Lemma \ref{LR01} and \ref{LR02}, we have
\begin{align} \label{ER11}
\lim_{d(x) \to +\infty} \frac{\text{Vol}_gB(p,d(x))}{w_{n}d^n} =\lim_{r(x) \to +\infty} \frac{\text{Vol}_{g_E}B(p,r(x))}{w_{n}r^n}=1.
\end{align}
Then from a corollary of Bishop-Gromov volume comparison theorem \cite[Corollary $1.134$]{CLN06}, we conclude that $(M,g)$ is isometric to $(\mathbb R^n,g_E)$.

\section{Perelman's $\mu$-functional}
Recall that Perelman's $\mathcal W$ entropy \cite{P03} is defined as
\begin{equation} \label{E201}
\mathcal W(g,f,\tau)=\int \left ( \tau(|\nabla f|^2+R)+f-n) \right ) \frac{e^{-f}}{(4 \pi \tau)^{n/2}}\,dV
\end{equation} 
for smooth function $f$ and $\tau >0$. Let $u=e^{-f/2}$, \eqref{E201} becomes
\begin{equation} \label{E202}
\overline{\mathcal W}(g,u,\tau)=\int \left ( \tau(4|\nabla u|^2+Ru^2)-u^2\log u^2-nu^2 \right ) (4 \pi \tau)^{-n/2} \,dV
\end{equation} 
Moreover, For a general (possibly incomplete) Riemannian manifold $(M,g)$, $\mu$-functional is defined as
\begin{equation} \label{E203}
\mu(g,\tau)= \inf\left\{\overline{\mathcal W}(g,u,\tau) \mid u \in W_0^{1,2}(M) \, \, \text{and} \, \int_M u^2(4 \pi \tau)^{-n/2} \,dV=1 \right\}.
\end{equation}

Note that when $M$ is complete, $W^{1,2}(M)=W_0^{1,2}(M)$. Moreover, from the definition we have $\mu_U(g,\tau) \ge \mu_M(g,\tau)$ for any open set $U \subset M$.

We have the following monotonicity result under Ricci flow for the complete noncompact manifold,
$$
\mu(g(t_2),\tau(t_2)) \ge \mu(g(t_1),\tau(t_1))
$$
for all $0\le t_1 \le t_2 < \bar \tau$ where $\tau(t)=\bar \tau-t,\, 0<\bar \tau<T$. Here we assume that Ricci flow exists for $[0,T]$ and $|\text{Rm}|$ is uniformly bounded in spacetime. The proof of the monotonicity formula can be found in \cite[Theorem $7.1$, (ii)]{CTY11}. Although in \cite{CTY11} they have only proved the case for the conjugate heat kernel, the same proof works for all $f$ which satisfies \cite[$(3.3)$, $(3.4)$]{P03}.

It is proved in \cite{STW04}, that $\mu(g,\tau)$ is finite if $g$ has bounded geometry, that is, the curvature is bounded and the injective radius is positive.
In particular, for any AE manifold the $\mu$-functional is finite.

Moreover, it is shown in \cite{Z11} that for a manifold with bounded geometry, $\overline{\mathcal W}(g,u,1)$ has a smooth positive minimizer if $\mu(g,1)$ is less than the corresponding value at infinity. Note that by our definition of $\overline{\mathcal W}$,
$$
\overline{\mathcal W}(g,u,1)=L(g,v)-\frac{n}{2}\log{4\pi}-n
$$
where the functional $L(g,v)$ is defined in \cite[$(1.1)$]{Z11} and $v=u(4\pi)^{-n/4}$. Therefore,
\begin{equation}\label{S001}
\mu(g,1)=\lambda(M)-\frac{n}{2}\log{4\pi}-n
\end{equation}
where, see \cite[Definition $1.1$]{Z11},
$$
\lambda(M)=\inf\left\{L(v,g) \, \vert \, \int_M v^2\, dV_g=1\right\}.
$$

To be more precise, if for any sequence $p_n \to \infty$ on the manifold $M$ such that $(M,g,p_n)$ converges smoothly in the Cheeger-Gromov sense to $(M_{\infty},g_{\infty},p_{\infty})$ and $\mu_M(g,1)<\mu_{M_{\infty}}(g_{\infty},1)$, 
then $\mu_M(g,1)$ has a smooth positive minimizer.

In the case of Euclidean space, it follows from log-Sobolev inequality of L. Gross \cite{G75} that
\begin{thm}
\begin{equation}\label{E2s01}
\mathcal W(g_E,f,\tau) \ge 0
\end{equation}
for any smooth $f$ such that $\int_{\mathbb R^n}e^{-f}(4\pi\tau)^{-n/2} \,dV_{g_E}=1$.
\end{thm}

The proof can be found in \cite[Lemma $8.17$]{PT06}.

It is immediate from \eqref{E2s01} that $\overline{\mathcal W}(g_E,u,\tau) \ge 0$ where equality holds if $u^2=e^{-\frac{|x|^2}{4\tau}}$. Therefore, $\mu_{\mathbb R^n}(g_E, \tau)=0$.
For an AE manifold $M^n$, we have $(M,g,p_n) \overset{C^{\infty}}{\longrightarrow} (\mathbb R^n, g_E, p_{\infty})$ for any sequence $p_n \to \infty$ by Cheeger-Gromov compactness theorem. Therefore $\overline {\mathcal W}(g,u,\tau)$ has a smooth positive minimizer if $\mu(g,\tau)=\mu(\tau^{-1}g,1)<0$ from the above result. Note that $\tau^{-1}g$ is still an AE metric.

We have the following lemma.

\begin{lem} \label{L4001}
Assume that $(M_i,g_i)$ converges to $(M_{\infty},g_{\infty})$ smoothly in the Cheeger-Gromov sense and  $\mu(g_{\infty},\tau)$ is finite, then
$$
\mu(g_{\infty},\tau) \ge \limsup_{i \to \infty} \mu(g_i,\tau).
$$
\end{lem}

\begin{proof}
For any $\epsilon>0$, we can find a $u \in W^{1,2}_0(M_{\infty})$ such that $\overline{\mathcal W}(g_{\infty},u,\tau) \le \mu(g_{\infty},\tau)+\epsilon$. For large $i$, we can find $u_i \in W^{1,2}_0(M_i)$ which are the pull-back functions of $u$ and $\lim_{i \to \infty}\overline{\mathcal W}(g_i,u_i,\tau)=\overline{\mathcal W}(g_{\infty},u,\tau)$ by the convergence.\\
Therefore we have
$$
\limsup_{i \to \infty}\mu(g_i,\tau) \le \lim_{i \to \infty}\overline{\mathcal W}(g_i,u_i,\tau) \le \mu(g_{\infty},\tau)+\epsilon.
$$
Since the above holds for any $\epsilon>0$, we have $\limsup_{i \to \infty} \mu(g_i,\tau) \le \mu(g_{\infty},\tau)$.
\end{proof}

It follows immediately from the above lemma that $\mu(g,\tau) \le 0$ for any AE manifold since  $(M,g,p_n) \overset{C^{\infty}}{\longrightarrow} (\mathbb R^n, g_E, p_{\infty})$ for any $p_n \to \infty$.

The Euler-Lagrange equation for the minimizer of $\mu(g,\tau)$ is 
\begin{equation} \label{E204}
\tau(-4\Delta u+Ru)-u \log {u^2}-nu=\mu (g,\tau)u.
\end{equation}

For the general Ricci flow on the noncompact manifold we have the following result and the proof is almost identical with the compact case, see \cite[Section $3.1$]{P03},
\begin{thm}\label{T201}
If $(M^n,g)$ is a manifold with bounded geometry such that a solution $g(t)$ of bounded curvature to the Ricci flow with $g(0)=g$ exists for $t \in [0,T)$, then for any $\bar{\tau} \in (0,T)$, $\mu(g,\bar{\tau})<0$ unless $(M^n,g)$ is isometric to $(\mathbb R^n,g_{E})$.
\end{thm}
\begin{proof}
Let $\tau(t)=\bar{\tau}-t$, $y \in M$ and consider the corresponding fundamental solution
\begin{equation} \label{E205}
v(x,t)=(4 \pi \tau(t))^{-n/2}e^{-f(x,t)},\qquad t \in [0,\bar{\tau})
\end{equation}
to the adjoint heat equation 
$$
\frac{\partial v}{\partial t}=-\Delta v+R v
$$
with $\lim_{t \nearrow \bar{\tau}}v(\cdot,t)=\delta_y$.

The existence of the fundamental solutions to the adjoint heat equation on noncompact manifolds and its basic properties can be found in \cite[Chapter $24$, $25$]{CCGGIIKLLN10a}.

Then by the monotonicity of the entropy,
\begin{equation} \label{E206}
\mu(g,\bar{\tau})=\mu(g,\tau(0)) \le \mathcal W(g(0),f(0),\tau(0)) \le \limsup_{t \nearrow \bar{\tau}} \mathcal W(g(t),f(t),\tau(t))\le 0
\end{equation}
where the proof of the last limit in \eqref{E206} can be found in \cite[Theorem $7.1$]{CTY11}.
If $\mu(g,\bar{\tau})=0$, $\mathcal W(g(t),f(t),\tau(t))=0$ since it is monotone. Therefore from the formula
\begin{equation} \label{mono}
\frac{d \mathcal W(g(t),f(t),\tau(t))}{dt}=2\tau \int_M \left|Rc+\nabla^2f-\frac{g}{2\tau}\right|^2 \frac{e^{-f}}{(4 \pi \tau)^{n/2}}\,dV
\end{equation}
we have 
\begin{equation}\label{E207}
Rc+\nabla^2f-\frac{g}{2\tau} \equiv 0
\end{equation}
for $t \in [0,\bar {\tau}]$, so $g(t)$ is a shrinking soliton with singular time $\bar {\tau}$. From 
$$
\tau(t)\underset{M}{\text{max}}|\text{Rm}(g(t))| \equiv \text{const}
$$ 
for $t \in [0,\bar {\tau}]$, we conclude that $|\text{Rm}(g(t))|\equiv 0$.
In particular $g$ is Ricci-flat and we have from \eqref{E207}
\begin{equation}\label{E208}
\nabla^2f-\frac{g}{2\bar{\tau}} \equiv 0.
\end{equation}
Set $\bar{f}=4\bar{\tau}f$, then $\nabla^2\bar{f}=2{g}$ and hence $\bar{f}$ is a convex function.

Let $O$ be a fixed point, then for any point $x \in M$ we have a minimizing geodesic $s(t) ,\, 0 \le t \le d(x,O)$ such that $|\dot{s}(t)|=1$.
Then we have
\begin{align}\label{E208a}
\frac{d^2\bar{f}(s(t))}{dt^2} &= \nabla^2\bar{f}(\nabla d, \nabla d)=2g(\nabla d, \nabla d)=2.
\end{align}
Therefore,
\begin{align} \label{S202}
 \frac{d\bar{f}(s(t))}{dt} =\langle \nabla \bar{f}, \nabla d \rangle=2t+\langle \nabla \bar{f}, \nabla d \rangle_{t=0}
\end{align}
From \eqref{S202} we have $\bar{f}(s(t))=\bar{f}(O)+t\langle \nabla \bar{f}, \nabla d \rangle_{t=0}+t^2$. In other words, $\bar{f}$ is quadratically increasing and therefore it has a minimal point $O_1$. By choosing $O=O_1$, we have $\bar{f}(x)=\bar{f}(O_1)+d^2(x,O_1)$.  In particular, by taking trace of \eqref{E208} we have 
$$
\Delta d^2=2n.
$$
Therefore $(M^n,g)$ is isometric to $(\mathbb R^n,g_{E})$ by Bishop-Gromov comparison theorem \cite[Theorem $1.128$, $1.132$]{CLN06} since $g$ is Ricci-flat.
\end{proof}

Now we have the following crucial result.  

\begin{thm}\label{T202}
If $(M^n,g)$ is an AE manifold such that the scalar curvature $R > 0$, then $\lim_{\tau \to \infty}\mu(g,\tau)=0$.
\end{thm}

\begin{proof}
If the conclusion does not hold, we can find a sequence $\tau_k \to +\infty$ and $\lim_{k \to \infty}\mu(g,\tau_k)=\mu_{\infty}$, so that $\mu_{\infty}$ is either a finite negative number or $\mu_{\infty}=-\infty$.

We have previously shown that $\mu(g,\tau_k)$ has a positive minimizer $u_k$ and it satisfies
\begin{equation} \label{E211}
\tau_k(-4\Delta u_k+Ru_k)-u_k \log {u_k^2}-nu_k=\mu (g,\tau_k)u_k 
\end{equation}
and 
\begin{equation} \label{E212}
\int_M u_k^2(4 \pi \tau_k)^{-n/2} \, dV=1 .
\end{equation}

\noindent \emph{Claim 1}.
$u_k$ are uniformly bounded.

We first prove a lemma.
\begin{lem}\label{L201}
For $u \in W^{1,2}(M)$, the following Sobolev inequality holds
\begin{equation}\label{E213}
\left( \int_M u^{\frac{2n}{n-2}} \, dV \right )^{\frac{n-2}{n}} \le C\int_M \left(4|\nabla u|^2+Ru^2\right) \, dV
\end{equation}
where the constant $C$ depends on the dimension, curvature bound, injective radius lower bound, AE coordinate system and infinimum of $R$ on a compact set.
\end{lem}

\begin{proof}
Let $M^n=K \bigsqcup E$ be the disjoint union of a compact set $K$ and AE end $E$ and $K_1$ a compact set such that $K \subset \subset K_1$.
We choose a cutoff function $\phi_0$ supported on $K_1$ and $\phi_0=1$ on $K$. Let $\phi_1=1-\phi_0$. 

For any $u \in W^{1,2}(M)$, we have
$$
\left\|u\right\|_{\frac{2n}{n-2}}=\left\|\phi_0 u+\phi_1u\right\|_{\frac{2n}{n-2}} \le \|\phi_0 u\|_{\frac{2n}{n-2}}+\|\phi_1 u\|_{\frac{2n}{n-2}}.
$$

By the $L^2$ Sobolev inequality on manifold with bounded geometry, see \cite[Theorem $2.21$]{Au10},
\begin{align} 
\left( \int_M \left(\phi_0 u\right)^{\frac{2n}{n-2}} \, dV \right )^{\frac{n-2}{n}} &\le
C  \int_M \left(|\nabla \left(\phi_0 u\right)|^2+\phi^2_0 u^2 \right)\, dV   \notag \\
& \le C  \int_{K_1} \left(|\nabla \phi_0 u|^2+|\phi_0 \nabla u|^2+\phi_0^2 u^2\right) \, dV \notag \\
& \le C \int_{K_1} \left(|\nabla u|^2+u^2 \right)\, dV \notag \\
& \le C \int_{K_1}\left(4|\nabla u|^2+Ru^2\right) \, dV.  \label{E214}
\end{align} 
The last inequality holds since we assume $R>0$.

On the $AE$ end $E$, by enlarging $K$ and $K_1$ if necessary, we can assume the $L^2$ Sobolev inequality of the Euclidean type holds.
To be precise, on $\mathbb R^n$ we have the $L^2$ Sobolev inequality \cite{A03}:
\begin{equation} \label{E215}
\left( \int_{\mathbb R^n}  u^{\frac{2n}{n-2}} \, dV_{g_E} \right)^{\frac{n-2}{n}}
\le C \int_{\mathbb R^n} |\nabla_{g_E} u|^2 \, dV_{g_E}
\end{equation}
for any $u \in C_0^{1}(\mathbb R^n)$ and some constant $C>0$ depending only on dimension.

Since $E$ is the AE end, by shrinking it if necessary, we can assume that there exists a $C>0$ such that
\begin{align}
C^{-1}dV_{g_E} &\le dV \le CdV_{g_E} \notag \\
C^{-1}|\nabla_{g_E} u|^2 &\le   |\nabla u|^2 \le C|\nabla_{g_E} u|^2. \notag 
\end{align}
Hence, for any $u \in C_0^1(E)$
\begin{align}
\left( \int_{E}  u^{\frac{2n}{n-2}} \, dV \right)^{\frac{n-2}{n}}&\le \left(C \int_{\mathbb R^n}  u^{\frac{2n}{n-2}} \, dV_{g_E} \right)^{\frac{n-2}{n}} \notag \\
&\le  C \int_{\mathbb R^n} |\nabla_{g_E} u|^2 \, dV_{g_E} \le C\int_{\mathbb R^n} |\nabla u|^2 \, dV_{g_E} \\ \notag
& \le C \int_{\mathbb R^n} |\nabla u|^2 \, dV\le C \int_{E} |\nabla u|^2 \, dV.
\end{align}

So we have
\begin{align} \label{E217}
\left( \int_M (\phi_1 u)^{\frac{2n}{n-2}} \, dV \right )^{\frac{n-2}{n}} &\le
C  \int_M |\nabla (\phi_1 u)|^2 \, dV   \notag \\
& \le C  \int_M \left(|\nabla \phi_1 u|^2+|\phi_1 \nabla u|^2\right) \, dV  \notag \\
& \le C  \int_M |\nabla u|^2 \, dV +C\int_{K_1}u^2 \, dV \notag \\
& \le C \int_M \left(4|\nabla u|^2+Ru^2 \right)\, dV. 
\end{align} 
Combining \eqref{E214} and \eqref{E217}, \eqref{E213} holds.
\end{proof}

We can now prove the claim by using the Moser iteration. This is known to experts but we write it down for the convenience of readers. For the sake of simplicity, we will not write down the subscript $k$ explicitly throughout and set $\mu=\mu(g,\tau_k)$.

\emph{Proof of Claim 1}, see also \cite[Lemma $2.1$]{Z11}.
From \eqref{E211} we have
$$
4\Delta u-Ru+\frac{2}{\tau}u\log u+\frac{n+\mu}{\tau}u = 0.
$$
Since $\mu \le 0$, we have
\begin{equation}\label{E218}
4\Delta u-Ru+\frac{2}{\tau}u\log u+\frac{n}{\tau}u \ge 0.
\end{equation}
By a direct computation, for $p \ge 1$
\begin{align} \label{E219}
4\Delta u^p &=4p(p-1)u^{p-2}|\nabla u|^2+4pu^{p-1}\Delta u \ge 4pu^{p-1}\Delta u \notag \\
&\ge -\frac{2p}{\tau}u^p \log u-\frac{np}{\tau}u^p+pRu^p.
\end{align}
We set $w=u^p$ and $\phi$ to be a test function. From \eqref{E219} we have 
$$
4\int \langle \nabla(w \phi^2), \nabla w \rangle \, dV \le \frac{2p}{\tau} \int w^2 \phi^2 \log u \, dV+ \frac{np}{\tau} \int w^2\phi ^2 \,dV-\int pRw^2\phi^2 \, dV.
$$
On the other hand, since
$$
\langle \nabla(w \phi^2), \nabla w \rangle=|\nabla(w\phi)|^2-|\nabla \phi|^2w^2
$$
we have
\begin{equation} \label{E220}
4\int |\nabla(w\phi)|^2 \, dV \le 4\int |\nabla \phi|^2w^2 \,d V+ \frac{2p}{\tau} \int w^2 \phi^2 \log u \, dV+ \frac{np}{\tau} \int w^2\phi ^2 \,dV-\int pRw^2\phi^2 \, dV.
\end{equation}
There is a constant $c_1 >0$ such that
$$
\log u \le u^{\frac{2}{n}}+c_1.
$$
Hence
\begin{align}\label{E221}
\frac{2p}{\tau} \int w^2 \phi^2 \log u \, dV & \le \frac{2p}{\tau} \int w^2 \phi^2 u^{\frac{2}{n}} \, dV+\frac{2c_1p}{\tau} \int w^2 \phi^2 \, dV \notag \\
&\le  \frac{2p}{\tau} \left( \int(w\phi)^{\frac{2n}{n-1}} \, dV \right)^{\frac{n-1}{n}} \left(\int u^2 \,dV \right)^{\frac{1}{n}}+\frac{2c_1p}{\tau}\int w^2 \phi^2 \, dV \notag \\
&= \frac{\sqrt{4\pi}2p}{\sqrt{\tau}}\left( \int(w\phi)^{\frac{2n}{n-1}} \, dV \right)^{\frac{n-1}{n}}+\frac{2c_1p}{\tau}\int w^2 \phi^2 \, dV
\end{align}
since \eqref{E212} holds.

From H\"older's inequality $\|fh\|_1\le \|f\|_p\|h\|_q$ by choosing $f=h=(w\phi)^{\frac{n}{n-1}}$, $p=\frac{2(n-1)}{n-2}$ and $q=\frac{2(n-1)}{n}$, we have
\begin{align}\label{E222}
\left( \int(w\phi)^{\frac{2n}{n-1}} \, dV \right)^{\frac{n-1}{n}} &\le \left( \int(w\phi)^{\frac{2n}{n-2}} \, dV \right)^{\frac{n-2}{2n}}\left(\int w^2 \phi^2 \, dV \right)^{\frac{1}{2}} \notag \\
&\le \lambda \left( \int(w\phi)^{\frac{2n}{n-2}} \, dV \right)^{\frac{n-2}{n}}+\frac{1}{4\lambda} \int w^2 \phi^2 \, dV,
\end{align}
where the last line is from Young's inequality for a positive $\lambda$ to be determined below.

So from \eqref{E221},
\begin{align} \label{E223}
\frac{2p}{\tau} \int w^2 \phi^2 \log u \, dV &\le \frac{c_2 \lambda p}{\sqrt{\tau}}\left( \int(w\phi)^{\frac{2n}{n-2}} \, dV \right)^{\frac{n-2}{n}} \notag \\
&+\frac{c_2p}{4\lambda \sqrt{\tau}}\int w^2 \phi^2 \, dV+\frac{2c_1p}{\tau}\int w^2 \phi^2 \, dV
\end{align}
where $c_2=2\sqrt{4\pi}$.

From lemma \eqref{L201}, \eqref{E220} \eqref{E223}, we have
\begin{align} \label{E224}
\frac{1}{C}\left( \int (w\phi)^{\frac{2n}{n-2}} \, dV \right )^{\frac{n-2}{n}} \le & \int (4|\nabla (w\phi)|^2+R(w\phi)^2 )\, dV \notag \\
\le & 4\int |\nabla \phi|^2w^2 \,d V+ \frac{2p}{\tau} \int w^2 \phi^2 \log u \, dV+ \frac{np}{\tau} \int w^2\phi ^2 \,dV \notag \\
\le & 4\int |\nabla \phi|^2w^2 \,d V+ \frac{c_2 \lambda p}{\sqrt{\tau}}\left( \int(w\phi)^{\frac{2n}{n-2}} \, dV \right)^{\frac{n-2}{n}} \notag \\
&+\frac{c_2p}{4\lambda \sqrt{\tau}}\int w^2 \phi^2 \, dV+\frac{2c_1p}{\tau}\int w^2 \phi^2 \, dV \notag \\
&+ \frac{np}{\tau} \int w^2\phi ^2 \,dV.
\end{align}

If we choose $\lambda$ satisfies $\dfrac{c_2 \lambda p}{\sqrt{\tau}}=\dfrac{1}{2C}$, that is, $\lambda =\dfrac{\sqrt{\tau}}{2Cc_2p}$, then from \eqref{E224},
there exists a $C_0>0$ such that 
\begin{equation}\label{E225}
\left( \int (w\phi)^{\frac{2n}{n-2}} \, dV \right )^{\frac{n-2}{n}} \le C_0 \int |\nabla \phi|^2w^2 \,d V+\frac{C_0 p^2}{\tau}\int w^2 \phi^2 \, dV.
\end{equation}

For any point $x$ on $M$, we choose $\phi_k$ such that it is supported on $B\left(x,\sqrt{\tau}(1+1/2^k) \right)$ and $\phi_k=1$ on $B\left(x,\sqrt{\tau}(1+1/2^{k+1}) \right)$ 
such that $|\nabla \phi_k| \le \dfrac{C2^k}{\sqrt{\tau}}$.

From \eqref{E225} we have

\begin{align} \label{E226}
\left( \int_{B(x,\sqrt{\tau}(1+1/2^{k+1}))} w^{\frac{2n}{n-2}} \, dV \right )^{\frac{n-2}{n}}  \le & \left( \int (w\phi_k)^{\frac{2n}{n-2}} \, dV \right )^{\frac{n-2}{n}} \notag \\
\le & C_0 \int |\nabla \phi_k|^2w^2 \,d V+\frac{C_0 p^2}{\tau}\int w^2 \phi_k^2 \, dV \notag \\
\le & \frac{C_1 2^{2k} p^2}{\tau} \int_{B(x,\sqrt{\tau}(1+1/2^k))} w^2 \, dV.
\end{align}

If we set $p_0=\dfrac{n}{n-2}$ and choose $p=p_0^k$, from \eqref{E226} we have
\begin{align} 
\left( \int_{B(x,\sqrt{\tau}(1+1/2^{k+1}))} u^{2p_0^{k+1}} \, dV \right )^{\frac{n-2}{n}} \le \frac{C_1(2p_0)^{2k}}{\tau} \int_{B(x,\sqrt{\tau}(1+1/2^k))} u^{2p_0^k} \, dV,
\end{align}
or equivalently,

\begin{align} \label{E227}
\left( \int_{B(x,\sqrt{\tau}(1+1/2^{k+1}))} u^{2p_0^{k+1}} \, dV \right )^{\frac{1}{p_0^{k+1}}} \le \frac{C_1^{\frac{1}{p_0^{k}}}(2p_0)^{\frac{2k}{p_0^{k}}}}{\tau^{\frac{1}{p_0^{k}}}} \left( \int_{B(x,\sqrt{\tau}(1+1/2
^k))} u^{2p_0^{k}} \, dV \right )^{\frac{1}{p_0^{k}}}.
\end{align}

Let $k=0,1,\ldots,$ and by iteration,
\begin{equation}\label{E228}
\max_{B(x,\sqrt{\tau})}u^2 \le \frac{C_1^{\sum_{k\ge 0}\frac{1}{p_0^{k}}}p_0^{\sum_{k \ge 0}\frac{2k}{p_0^{k}}}}{\tau^{\sum_{k \ge 0}\frac{1}{p_0^{k}}}} \left( \int_{B(x,2\sqrt{\tau})} u^2 \, dV \right )
\le \frac{C_2}{\tau^{\frac{n}{2}}} \left( \int_{B(x,2\sqrt{\tau})} u^2 \, dV \right ) 
\end{equation}
since $\sum_{k \ge 0}\frac{1}{p_0^{k}}=\frac{n}{2}$ and $\sum_{k \ge 0}\frac{2k}{p_0^{k}}$ converges.
As 
$$
\int_{B(x,2\sqrt{\tau})} u^2 \, dV  \le \int_M u^2 \, dV=(4 \pi \tau)^{\frac{n}{2}},
$$
we conclude from \eqref{E228} that 
$$
\max_M u^2 \le C_3
$$
for some constant $C_3>0$. 

Hence all $u_k$ are uniformly bounded. 

Since every minimizer is exponentially decaying, see \cite[Lemma $2.3$]{Z11}, there is a maximum point $p_k$ for $u_k$. Since $\Delta u_k(p_k) \le 0$, at $p_k$ we have in \eqref{E211}
$$
\tau_k Ru_k-u_k \log {u_k^2}-nu_k-\mu_k u_k \le 0. 
$$
As $u_k>0$, we have
$$
u_k(p_k)\ge \exp \left( \frac{R(p_k)\tau_k-n-\mu_k}{2} \right) \ge \exp \left( \frac{-n-\mu_k}{2} \right).
$$
As we have proved that $u_k$ is uniformly bounded, $\mu_k$ cannot tend to $-\infty$. In other words, $\mu_{\infty}$ is finite.

From \eqref{E212} we have
$$
\int_K u_k^2 \, dV+  \int_E u_k^2 \, dV=(4 \pi \tau_k)^{\frac{n}{2}}.
$$
Since $u_k$ are uniformly bounded and $K$ has finite volume, the first integral is uniformly bounded. Hence there is a $c_0 \in (0,1]$ satisfying
\begin{equation}\label{E2301}
 \int_{E} u_k^2 \, dV  \ge c_0(4 \pi \tau_k)^{\frac{n}{2}}.
\end{equation}

We define functions $\tilde u_k(x)=u_k(\sqrt{\tau_k}x)$, a new metric on $E$ as $\tilde g_{ij}(x)=g_{ij}(\sqrt{\tau_k}x)$, the corresponding Laplace operator $\widetilde \Delta_k=\dfrac{1}{\sqrt{\det \tilde g}}\partial_i \sqrt{\det \tilde g} \tilde g^{ij} \partial_j$ and scalar curvature $\tilde R(x)=\dfrac{1}{\tau_k}R(\sqrt{\tau_k}x)$. 

The metric $\tilde g$ on $E$, after a diffeomorphism, is nothing but $\tau_k^{-1} g$. So by the AE condition, $(E, \tilde g)$ converges in the Cheeger-Gromov sense to $(\mathbb R^n \backslash \{0\}, g_E)$ and the convergence is smooth away from the origin.

Now \eqref{E211} becomes
\begin{equation} \label{E231}
-4\widetilde \Delta_k \tilde u_k+\tilde R \tilde u_k-\tilde u_k \log {\tilde u_k^2}-n\tilde u_k=\mu_k \tilde u_k 
\end{equation}

All $\tilde u_k$ can be regarded as functions defined on $\mathbb R^n$ except for a ball with center $0$. We next prove that there is a limit in $W^{1,2}(\mathbb R^n)$ for the sequence $\{\tilde u_k\}$.

Since $\mu_k$ are bounded, from \eqref{E211} and $\eqref{E212}$ we have, for details see \cite[$(29)$]{STW04},
\begin{equation} \label{E232}
\tau_k \int_M |\nabla u_k|^2 (4 \pi \tau_k)^{-n/2} \, dV \le C
\end{equation}
where the bound $C$ is independent of $k$.

Therefore, for any annulus $C_{a,A}=\{x \in \mathbb R^n \mid a<|x|<A \}$, we have a uniform constant $C_1 >0$ such that
$$
\int_{C_{a,A}} \tilde u_k^2 \,d\widetilde V \le C_1
$$
and
$$
\int_{C_{a,A}} |\widetilde \nabla \tilde u_k|^2 \, d\widetilde V \le C_1
$$
for $k$ sufficiently large.

In other words, $\tilde u_k$ are bounded in $W^{1,2}(C_{a,A})$ and hence a subsequnce of $\{\tilde u_k\}$ converges weakly to a function $u_{\infty}$ in $W^{1,2}(C_{a,A})$ and by Sobolev immbedding converges strongly to $u_{\infty}$ in $L^p(C_{a,A})$ if $1 \le p < 2n/{n-2}$.
Choosing two sequences $a_m \to 0$ and $A_m \to \infty$ for $m=1,2,\ldots$, by the diagonal argument replacing $\{\tilde u_k \}$ by a subsequence if necessary, we have a function $u_{\infty}$ defined on $\mathbb R^n \backslash \{0\}$ such that
for every compact set $C$ in $\mathbb R^n \backslash \{0\}$, there is an $N>0$ such that $\{\tilde u_k ,\, k \ge N\}$ converges weakly to $u_{\infty}$ in $W^{1,2}(\mathbb R^n \backslash \{0\})$ and strongly in $L^p(\mathbb R^n \backslash \{0\})$ if $1 \le p < 2n/{n-2}$.

By the standard $L^p$ regularity property of elliptic equation \eqref{E231}, see \cite[Theorem $9.11$]{GT01}, the convergence is in $C^{1,\alpha}_{\text{loc}}(\mathbb R^n \backslash \{0\})$ for some $\alpha >0$. Therefore if $k \to \infty$ in \eqref{E231}, we have
\begin{equation} \label{E233}
-4\Delta_{g_E} u_{\infty}-u_{\infty} \log {u_{\infty}^2}-nu_{\infty}=\mu_{\infty} u_{\infty}.
\end{equation}
By the standard regularity property of elliptic operator and bootstrapping, see \cite[Theorem $6.17$]{GT01}, we know that $u_{\infty} \in C^{\infty}(\mathbb R^n \backslash \{0\})$ and either $u_{\infty} \equiv 0$ or $u_{\infty} >0$ by the strong maximum principle \cite{R81}.

Moreover we have 
\begin{equation} \label{E234}
\int_{\mathbb R^n \backslash \{0\}} u_{\infty}^2 \,dV_{g_E} \le (4\pi)^{\frac{n}{2}},
\end{equation}
and there exists a $C>0$ such that
\begin{equation} \label{E235}
\int_{\mathbb R^n \backslash \{0\}} |\nabla u_{\infty}|^2 \,dV_{g_E} \le C.
\end{equation}

\noindent \emph{Claim 2}.
$u_{\infty} \in W^{1,2}(\mathbb R^n)$. 

\emph{Proof of Claim 2}.
We first prove a lemma.

\begin{lem} \label{L202}
For a function $f \in C^1(\mathbb R^n \backslash \{0\})$, if $|f(x)| \le C|x|^{-\alpha}$ for some $\alpha <n-1$ and small $x$ and $|\nabla f|$ is integrable on the punctured ball $B(0,1)\backslash \{0\}$, then the function
$$
\tilde f(x)=
\begin{cases}
f(x), & x \ne 0; \\ 
0, & x=0.
\end{cases}
$$
has the weak derivative
$$
g_i(x)=
\begin{cases}
\partial_i f(x), & x \ne 0; \\ 
0, & x=0.
\end{cases}
$$
for $i=1,2,\ldots,n$
\end{lem}

\begin{proof}
For any $\phi \in C_0^{\infty}(\mathbb R^n)$,

\begin{align}
\int_{\mathbb R^n} \tilde f \partial_i \phi \, dV_{g_E} &= \lim_{r \to 0} \int_{\mathbb R^n \backslash B(0,r)} f \partial_i \phi \, dV_{g_E} \notag \\
& =-\lim_{r \to 0}\int_{\mathbb R^n \backslash B(0,r)} \partial_i f \phi \, dV_{g_E} +\lim_{r \to 0}\int_{S(0,r)} f \phi \upsilon^i \, d\sigma \notag \\
&=-\int_{\mathbb R^n} g_i \phi \, dV_{g_E} +\lim_{r \to 0}\int_{S(0,r)} f \phi \upsilon^i \, d\sigma \notag
\end{align}
where $\upsilon^i$ is the $i$th component of the inner normal vector of $S(0,r)$. The first integral in the last line is finite since $g_i$ is integrable by our assumption.

From the condition,
$$
\left| \int_{S(0,r)} f \phi \upsilon^i \, d\sigma \right| \le C'r^{n-1} \underset {x \in S(0,r)}{\max}|f| \le C'Cr^{n-1-\alpha}. 
$$

Since $\alpha <n-1$ we conclude that 
$$
\lim_{r \to 0}\int_{S(0,r)} f \phi \upsilon^i \, d\sigma=0
$$
and the lemma follows.
\end{proof}

Applying Moser's iteration to \eqref{E233} as the proof of Claim $1$, we have for any $0 < r \le 1$ and $|p|=r$,
$$
\underset{B(p,r/4)}{\max} u_{\infty}^2 \le \frac{C}{r^n} \int_{B(p,r/2)} u_{\infty}^2 \, dV_{g_E} \le \frac{C'}{r^n}.
$$
Hence we have
$$
u_{\infty}(x) \le \frac{C}{|x|^{n/2}}
$$
for $|x| \le 1$.
Therefore, by combining \eqref{E235} we can apply Lemma \ref{L202} to conclude that $u_{\infty}$ can be extended to $\mathbb R^n$. Moreover from \eqref{E234} and \eqref{E235}, 
$u_{\infty} \in W^{1,2}(\mathbb R^n)$. \\

\emph{Case 1}: $u_{\infty}>0$.

From \eqref{E234} we have
$$
0<\int_{\mathbb R^n} u_{\infty}^2 (4 \pi)^{-n/2}\,dV_{g_E} =c_1^2 \le 1.
$$
So if we set $\tilde u_{\infty}=u_{\infty}/c_1$, from \eqref{E233} we have
\begin{align}
&\int_{\mathbb R^n} (4|\nabla \tilde u_{\infty}|^2-\tilde u_{\infty}^2 \log{\tilde u_{\infty}^2}-n\tilde u_{\infty}^2)(4 \pi)^{-n/2} \, dV_{g_E} \notag \\
=&\frac{1}{c_1^2}  \int_{\mathbb R^n} (4|\nabla u_{\infty}|^2-u_{\infty}^2 \log{u_{\infty}^2}-nu_{\infty}^2)(4 \pi)^{-n/2} \, dV_{g_E} +\log{c_1^2} \notag \\
=&\mu_{\infty}+\log{c_1^2}<0
\end{align}
since $\mu_{\infty}<0$ and $c_1^2<1$.
But it contradicts the fact that $\mu_{\mathbb R^n}(g_E,1)=0$.\\

\emph{Case 2}: $u_{\infty} \equiv 0$.

In this case it means that $\tilde u_k(x)=u_k(\sqrt{\tau_k} x)$ converges uniformly to $0$ on any compact set of $E$.

We can assume that 
$$\limsup_{k \to \infty}\max _{x\in \mathbb R^n \backslash B(0,1)}\tilde u_k(x) =0.
$$
Otherwise, if there exists a sequence $\{p_k\}_{k \in \mathbb N}$ such that $\tilde u_k(p_k) \ge c>0$, by our assumption $p_k \to \infty$. On the other hand, $(M, \tilde g_k,p_k)$ converges smoothly to $(\mathbb R^n, g_E,p_{\infty})$ and hence $\tilde u_k(x)$ converges to $u'_{\infty}$ which is not identically zero. Then like case $1$, we have a contradiction.

Choose a small constant $a>0$ such that 
\begin{equation} \label{AE_12}
\int_{E \backslash B(0,2a\sqrt{\tau_k})} u_k^2 \, dV \ge \frac{c_0}{2} (4\pi \tau_k)^{\frac{n}{2}}.
\end{equation}
This is possible since \eqref{E2301} holds and $u_k$ are uniformly bounded.

Choose a function $\phi$ such that $\phi \in C_0^{\infty}(\mathbb R^n \backslash B(0,a))$ and $\phi =1$ on $\mathbb R^n \backslash B(0,2a)$. Then we have, like \eqref{E220}
\begin{align} \label{E238}
&\int \left(4|\widetilde \nabla (\phi \tilde u_k)|^2+(\tilde R-n) (\phi \tilde u_k)^2- (\phi \tilde u_k)^2 \log \tilde u_k^2 \right) (4\pi)^{-n/2} \, d \widetilde V \notag \\
= & \int 4|\widetilde \nabla \phi|^2 \tilde u_k^2 (4\pi)^{-n/2}\, d \widetilde V+ \mu_k \int (\phi \tilde u_k)^2  (4\pi)^{-n/2}\, d\widetilde V \\ \notag 
 \le& C\int_{C_{a,2a}} \tilde u_k^2 (4\pi)^{-n/2}\, d \widetilde V+\mu_k \int (\phi \tilde u_k)^2  (4\pi)^{-n/2}\, d\widetilde V.
\end{align}
But from our assumption $\{\tilde u_k\}$ converges to $0$ uniformly on $C_{a,2a}$, there exists a sequence $\{\epsilon_k\} \searrow 0$ such that
\begin{equation} \label{E240}
\int \left(4|\widetilde \nabla (\phi \tilde u_k)|^2+(\tilde R-n) (\phi \tilde u_k)^2- (\phi \tilde u_k)^2 \log (\phi \tilde u_k)^2 \right) (4\pi)^{-n/2}\, d\widetilde V \le \epsilon_k+\mu_k \int (\phi \tilde u_k)^2 (4\pi)^{-n/2}  \, d\widetilde V
\end{equation}
if $k$ is sufficiently large.

On the other hand,
$$
(4\pi)^{\frac{n}{2}} \ge \int \tilde u_k^2  \, d\widetilde V  \ge \int (\phi \tilde u_k)^2  \, d\widetilde V \ge \int_{\mathbb R^n \backslash B(0,2a)}\tilde u_k^2  \, d\widetilde V \ge \frac{c_0}{2}(4\pi)^{\frac{n}{2}}.
$$

So if we set 
$$
\int (\phi \tilde u_k)^2  \, d\widetilde V=\eta_k^2(4\pi)^{\frac{n}{2}},
$$
and $\psi_k=\dfrac{\phi \tilde u_k}{\eta_k}$, then $\eta_k \in [\frac{c_0}{2}, 1]$ and 
$$
\int \psi_k^2  (4\pi)^{-n/2}\, d\widetilde V=1.
$$
From \eqref{E240} we have,
\begin{align}\label{S100}
&\int \left( 4|\widetilde \nabla \psi_k|^2+(\tilde R-n) \psi_k^2- \psi_k^2 \log \psi_k^2 \right)(4\pi)^{-n/2}  \, d\widetilde V  \notag \\
\le & \eta_k^{-2} \epsilon_k+\mu_k+\log \eta_k^2 \le \eta_k^{-2} \epsilon_k+\mu_k \le 4c_0^{-2}\epsilon_k+\mu_k. 
\end{align}
When $k$ is sufficiently large, $4c_0^{-2}\epsilon_k+\mu_k$ is negative. Since $\psi_k$ converges to $0$ uniformly on $\mathbb R^n$, it is easy to check that $4|\widetilde \nabla \psi_k|^2+(\tilde R-n) \psi_k^2- \psi_k^2 \log \psi_k^2$ is positive when $k$ is large.

Thus we have derived a contradiction and the proof of Theorem \ref{T202} is complete. 
\end{proof}

With the same proof as Theorem \ref{T202}, we have the following uniform version which will be used in Section $7$.

\begin{thm}\label{TC202}
Let $(M^n_i,g_i)$ be a family of AE manifolds of the same order $\sigma >0$ with positive scalar curvature. For some compact sets $K_i \subset M^n_i$, we have a family of diffeomorphisms $\Phi_i: M_i^n \backslash K_i \to \mathbb R^n \backslash B(0,A)$ such that under these identifications,
\begin{equation}
|(g_i)_{uv} -\delta_{uv}| \le C_0r^{-\sigma_i}, \quad | \partial^{|k|}(g_i)_{uv}| \le C_kr^{-\sigma-k}, \quad 1\le u,v \le n
\end{equation} 
for some constants $C_k, k=0,1,\ldots$ which are independent of $i$.
Moreover, there exist compact sets $K_i'$ containing $K_i$ such that $\text{dis}_{g_E}(K_i,K_i')\ge d_0$ and 
$$
\left(\int_{M_i-K_i'}u^{\frac{2n}{n-2}}dV\right)^{\frac{n-2}{n}} \le C\int_{M_i-K_i'}|\nabla u|^2dV
$$
for some $d_0>0, \,C>0$ and any $u \in C^1_0(M_i-K_i')$. In addition, if $|\text{Rm}|_{g_i} \le R_0$, $\text{inj}_{g_i} \ge i_0$, $\text{Vol}_{g_i}(K_i')\le V_0$ and $\inf_{p \in K_i'} R_{g_i}(p) \ge r_0$ for some positive constants $R_0, r_0, i_0$ and $V_0$, we have
$$
\lim_{\tau \to +\infty} \mu_{M_i}(g_i,\tau)=0
$$
for all $g_i$ uniformly.
\end{thm}

\begin{rem}
We can get a uniform constant for Lemma \ref{L201} since the Sobolev constant only depends on the bounds of curvature and injective radius. The volume control of $K_i'$ is used to prove \eqref{E2301}.
\end{rem}

Next, we use Theorem \ref{T202} to prove the no local collapsing theorem in the case of AE manifold. Recall that a Riemannian manifold is $\kappa$-noncollapsed on all scales if for any metric ball $B(x,r)$ satisfying $|\text{Rm}|\le r^{-2}$ for all $y \in B(x,r)$, we have
$$
\frac{\text{Vol}B(x,r)}{r^n}\ge \kappa.
$$

Following the celebrated work of Perelman, we have

\begin{thm}\label{T205}
Let $g(t)$, $t\in [0,\infty)$, be the Ricci flow solution on an AE manifold $M^n$ with $R>0$, then there exists a $\kappa >0$ such that $g(t)$ is $\kappa$-noncollapsed on all scales.
\end{thm}

\begin{proof}
Since Ricci flow preserves the AE condition. So there exists a $\kappa_1>0$ such that for any $t \in [0,1]$, $r>0$, we have

\begin{equation}\label{E242}
\frac{\text{Vol}B_{g(t)}(x,r)}{r^n}\ge \kappa_1,
\end{equation}
where $B_{g(t)}(x,r)$ is a metric ball in $(M^n,g(t))$.

For $t \in [1,\infty)$, $r>0$ and $p \in M$ such that $|\text{Rm}|\le r^{-2}$ in $B_{g(t)}(x,r)$ we have the following inequality whose proof can be found in \cite[Proposition $5.37$]{CLN06}
\begin{equation}\label{E243}
\mu(g(t),r^2) \le \log \frac{\text{Vol}B_{g(t)}(x,r)}{r^n} +C(n).
\end{equation}

Then by \eqref{E243}, Theorem \ref{T202} and the continuity and monotonicity of $\mu(g,\tau)$, there exists a constant $C$ depending on $g(0)$ that
$$
C \le \mu(g(0), r^2+t) \le \mu(g(t),r^2) \le \log \frac{\text{Vol}B_{g(t)}(x,r)}{r^n} +C(n).
$$

We conclude that there exists $\kappa_2>0$ such that
\begin{equation}\label{E244}
\frac{\text{Vol}B_{g(t)}(x,r)}{r^n}\ge \kappa_2.
\end{equation}

Combining \eqref{E242} and \eqref{E244}, we can find $\kappa=\min(\kappa_1,\kappa_2)>0$ such that $g(t)$ is $\kappa$-noncollapsed on all scales. 
\end{proof}

\section{Analysis of singularity at time infinity}

For the Ricci flow $(M,g(t))$, $t \in [0,\infty)$, there are two different types of singularity at infinity classified by Hamilton, see \cite{H95}.

Case $1$ (Type IIb): $\text{sup}_{M \times [0,\infty)} t|\text{Rm}|=\infty$.

In this case, we take any sequences of times $T_i \to \infty$ and then choose $p_i=(x_i,t_i) \in M^n \times [0,T_i]$ such that
\begin{align}\label{S201}
t_i(T_i-t_i)|\text{Rm}|(x_i,t_i)=\underset{M^n \times (0,T_i]}{\text{sup}}t(T_i-t)|\text{Rm}|(x,t).
\end{align}
It can be seen from the above choice that $t_i \to +\infty$. Indeed, from the definition of Type IIb, we can find two sequences $L_i \to +\infty$, $y_i \in M$ such that $\lim_{i \to +\infty} L_i |\text{Rm}|(y_i,L_i)=+\infty$ and $L_i \le T_i/2$. Then we have
\begin{align}\label{S201a}
\underset{M^n \times (0,T_i]}{\text{sup}}t(T_i-t)|\text{Rm}|(x,t) \ge L_i(T_i-L_i)|\text{Rm}|(y_i,L_i) \ge \frac{1}{2}T_iL_i |\text{Rm}|(y_i,L_i).
\end{align}
Then it is clear from \eqref{S201} and \eqref{S201a} that $t_i \to +\infty$.

If we set $Q_i=|\text{Rm}|(x_i,t_i)$, it can be proved that $( M, g_i(t),p_i)$ converges smoothly in the Cheeger-Gromov sense to a complete eternal Ricci flow solution $(M_{\infty},g_{\infty}(t),p_{\infty}) ,\, t \in(-\infty, +\infty)$ where $g_i(t)=Q_ig(t_i+Q_i^{-1}t)$.

Then for any $\tau>0$, 
\begin{align}\label{E303}
\mu(g_{\infty}(0),\tau) &\ge \limsup_{i \to \infty} \mu(Q_ig(t_i),\tau) \notag \\
&\ge \limsup_{i \to \infty} \mu(g(t_i),\frac{\tau}{Q_i})\notag \\
&\ge \limsup_{i \to \infty} \mu(g(0),\frac{\tau}{Q_i}+t_i)=0 
\end{align}
where the first inequality follows from Lemma \ref{L4001}, the last from the monotonicity of $\mu$ and the equality is from Theorem \ref{T202}.

From Theorem \ref{T201}, it must be the case that $M^n$ is isometric to $\mathbb R^n$. But this is impossible since $|\text{Rm}|_{g_{\infty}(0)}(x_{\infty})=\lim_{i \to \infty}|\text{Rm}|_{g_i(0)}(x_i)=1$.

Case $2$ (Type III): $\text{sup}_{M \times [0,\infty)} t|\text{Rm}|<\infty$.

In this case, suppose $p_i=(x_i,t_i)$ is a sequence of points and times with $t_i \to \infty$ and 
$$
t_i|\text{Rm}|(x_i,t_i)=t_i\underset{x\in M}{\text{sup}}|\text{Rm}|(x,t_i) \ge c
$$ 
for some $c>0$. Then like the first case $(M,g_i(t)=Q_ig(t_i+Q_i^{-1}t),p_i)$, $t \in [-t_iQ_i,\infty)$, converges to $(M_{\infty},g_{\infty}(t),x_{\infty}) ,\, t \in(-c, +\infty)$, where $g_i(t)=Q_ig(t_i+Q_i^{-1}t)$. Again we derive a contradiction.

Therefore, we have proved that the singularity at infinity is of type III, and
\begin{equation}\label{E302}
\lim_{t \to \infty} t \,\underset{M}{\text{sup}}|\text{Rm}(t)|=0.
\end{equation}

We choose an $\epsilon \in (0,1)$ to be determined later. From \eqref{E302} we assume for $t$ large enough,

\begin{equation}\label{E303}
\underset{M}{\text{sup}}|\text{Rm}| \le \frac{\epsilon}{1+t}. 
\end{equation}

So by a translation of time, we assume \eqref{E303} holds for any $t \ge 0$.

Next, we prove a gradient estimate and Harnack inequality for the solution of heat equation under the condition of \eqref{E303}. The proof is a long time version of the Li-Yau estimates, see \cite{LY86}.

Set $u_0=r^{-2-\sigma}$ where $r$ is the function defined in the introduction. We consider the positive solution $u$ of the heat equation

\begin{equation}\label{E06}
u_t=\Delta u
\end{equation}
with the initial condition $u(0)=u_0$.

It can be proved by using the maximum principle as in the proof of Theorem \ref{R01}, that for any $T>0$, $t \in [0,T]$, $u(t)$ and $|\nabla u|(t)$ have the same decaying rates as $u(0)$ and $|\nabla_{g(0)}u|(0)$, respectively. To be precise, there exist $c_1(T)>0$ and $c_2(T)>0$ such that 

\begin{align}\label{E305}
c_1(T)r^{-2-\sigma} \le u&(t) \le c_2(T)r^{-2-\sigma}, \\  \notag
c_1(T)r^{-3-\sigma} \le|\nabla u&|(t) \le c_2(T)r^{-3-\sigma}.
\end{align}

Let $f=\log u$. Then $f$ satisfies
$$
f_t=\Delta f+|\nabla f|^2.
$$
If we set $H(x,t)=t(|\nabla f|^2-2f_t)$, then we have the following lemma.

\begin{lem} \label{L303}
Under the condition $\underset{M}{\text{sup}}\,|\text{Rm}|(x,t) \le \dfrac{\epsilon}{1+t}$,
\begin{align}\label{E306}
\Delta H-H_t \ge -2\nabla f \cdot \nabla H+\frac{t}{n}(|\nabla f|^2-f_t)^2-(|\nabla f|^2-2f_t)-3|\nabla f|^2-\frac{4\epsilon^2}{1+t} \\
\end{align}
\end{lem}

\begin{proof}
We have
\begin{align} \label{E307}
\Delta H=t\Delta (|\nabla f|^2-2f_t).
\end{align}

By using the Bochner's formula
\begin{align}
\Delta |\nabla f|^2&=2|\nabla ^2 f|^2+2Rc(\nabla f, \nabla f)+2\langle \nabla\Delta f,\nabla f \rangle \\ \notag
& =2|\nabla ^2 f|^2+2Rc(\nabla f, \nabla f)-2\langle \nabla(|\nabla f|^2-f_t),\nabla f \rangle \\ \notag
& \ge 2|\nabla ^2 f|^2-\frac{2}{1+t}|\nabla f|^2-2\langle \nabla(|\nabla f|^2-f_t),\nabla f \rangle
\end{align}
where the last inequality follows from our curvature estimate.

On the other hand,
$$
\Delta f_t=(\Delta f)_t-2R_{ij}f_{ij} \le (\Delta f)_t+ 2|Rc|^2+\frac{1}{2}|\nabla^2f|^2.
$$

So we get
\begin{align} \label{E308}
\Delta H &\ge t\left(|\nabla ^2 f|^2-2\langle \nabla(|\nabla f|^2-f_t),\nabla f \rangle-2(\Delta f)_t-\frac{2}{1+t}|\nabla f|^2-4|Rc|^2 \right) \notag \\
&\ge \frac{t}{n}(|\nabla f|^2-f_t)^2-2t\langle \nabla(|\nabla f|^2-f_t),\nabla f \rangle \notag \\
& \quad  +2t(|\nabla f|^2-f_t)_t-2|\nabla f|^2-\frac{4\epsilon ^2}{1+t}.
\end{align}

Then we have
$$
H_t=|\nabla f|^2-2f_t+t(|\nabla f|^2-2f_t)_t.
$$

Therefore,
\begin{align} \label{E309}
\Delta H-H_t \ge & \frac{t}{n}(|\nabla f|^2-f_t)^2-2t\langle \nabla(|\nabla f|^2-f_t),\nabla f \rangle \notag \\
& +2t(|\nabla f|^2-f_t)_t-t(|\nabla f|^2-2f_t)_t-(|\nabla f|^2-2f_t)-2|\nabla f|^2-\frac{4\epsilon ^2}{1+t}  \notag \\
=& \frac{t}{n}(|\nabla f|^2-f_t)^2-2t\langle \nabla(|\nabla f|^2-f_t),\nabla f \rangle \notag \\
& +t|\nabla f|_t^2-(|\nabla f|^2-2f_t)-2|\nabla f|^2-\frac{4\epsilon ^2}{1+t}  \notag \\
=& \frac{t}{n}(|\nabla f|^2-f_t)^2-2t\langle \nabla(|\nabla f|^2-f_t),\nabla f \rangle \notag \\
& +2t\langle \nabla f_t,\nabla f \rangle+2tRic(\nabla f,\nabla f)-(|\nabla f|^2-2f_t)-2|\nabla f|^2-\frac{4\epsilon ^2}{1+t}  \notag \\
\ge& \frac{t}{n}(|\nabla f|^2-f_t)^2-2\langle \nabla H,\nabla f \rangle-(|\nabla f|^2-2f_t)-3|\nabla f|^2-\frac{4\epsilon ^2}{1+t}.  
\end{align}
\end{proof}

Now we can use the above equation to derive the Li-Yau inequality by following the same method in \cite[Theorem $4.2$]{SY94} to conclude that
\begin{equation}\label{E3100}
\frac{|\nabla u|^2}{u^2}-2\frac{u_t}{u}\le \frac{c_1}{t}
\end{equation}
for some $c_1>0$.
Note that in \cite[($1.10$)]{SY94} the extra term $2nk$ when $\alpha=2$ can be bounded by $\dfrac{C}{1+t}$ in our case.

With the gradient estimate \eqref{E3100}, we prove the following Harnack inequality for $u$.

\begin{thm} \label{T307}
For any $x,y \in M^n$ and $0<t_1<t_2$,
$$
\frac{u(y,t_2)}{u(x,t_1)}\ge \left(\frac{t_2}{t_1}\right)^{-c_1/2} \exp \left(-\frac{d_{g(t_1)}(x,y)^2}{2(t_2-t_1)}(1+t_2-t_1)^{2\epsilon} \right).
$$
\end{thm}

\begin{proof}
Suppose $\gamma (t):[t_1,t_2] \to M$ is a geodesic with respect to the metric $g(t_1)$ such that 
$$
|\dot{\gamma}(t)|=\frac{d_{g(t_1)}(x,y)}{t_2-t_1}, \quad t_1\le t \le t_2, \\
$$
$$
\gamma (t_1)=x, \quad \gamma (t_2)=y.
$$
Then we have
\begin{align} \label{E311}
\log \frac{u(y,t_2)}{u(x,t_1)}=& \int_{t_1}^{t_2} \frac{d}{dt}\left(\log u(\gamma (t),t)\right) \, dt \notag \\
=& \int_{t_1}^{t_2} \left( \frac{\partial}{\partial t} \log u +\nabla \log u \cdot \frac{\partial \gamma}{\partial t} \right ) \, dt \notag \\
\ge & \int_{t_1}^{t_2} \left(\frac{|\nabla \log u|^2}{2}-\frac{c_1}{2t} +\nabla \log u \cdot \frac{\partial \gamma}{\partial t} \right ) \, dt \notag  \quad \text{using}\, \eqref{E3100}\\
\ge & -\frac{c_1}{2}\log \left(\frac{t_2}{t_1}\right)-\frac{1}{2}\int_{t_1}^{t_2} \left|\frac{\partial \gamma}{\partial t}\right|_{g(t)}^2 \, dt.
\end{align}

Using the evolution equation of metric along Ricci flow and inequality \eqref{E303},
$$
\int_{t_1}^{t_2} \left|\frac{\partial \gamma}{\partial t}\right|_{g(t)}^2 \, dt \le (1+t_2-t_1)^{2\epsilon}\int_{t_1}^{t_2} \left|\frac{\partial \gamma}{\partial t}\right|_{g(t_1)}^2 \, dt
=(1+t_2-t_1)^{2\epsilon}\frac{d_{g(t_1)}(x,y)^2}{t_2-t_1}
$$
from the estimate \eqref{E303}.

Therefore \eqref{E311} completes the proof.
\end{proof}

\begin{rem}
We note that the proof of the above estimates does not depend on the order of decaying for the initial condition $u_0$. 
\end{rem}

\begin{thm}\label{T308} We have the following estimate.
There exist $\delta >0$ and $C>0$ such that
$$
u(x,t) \le \dfrac{C}{(1+t)^{1+\delta}}.
$$ 
\end{thm}

\begin{proof}
We fix a constant $p \in (\frac{n}{2+\sigma},\frac{n}{2})$, then from the decaying property \eqref{E305} $u^p$ is integrable and

\begin{align}\label{E13}
\frac{d}{dt}\left( \int u^p \, dV \right) =&  \int ( pu^{p-1}u_t-Ru^p) \, dV  \le \int pu^{p-1}\Delta u \, dV  \notag  \\ 
=&\lim_{r \to +\infty}\int_{r(x)=r}pu^{p-1}\langle\nabla u,\nabla r \rangle \, d\sigma-\lim_{r \to +\infty}\int_{r(x)\le r} p(p-1)u^{p-2}|\nabla u|^2 \, dV \notag  \\ 
= & -\int p(p-1)u^{p-2}|\nabla u|^2 \, dV \le 0 
\end{align}
where the boundary term from the integration by parts vanishes since
\begin{align}\label{S202}
|\nabla u|u^{p-1} \le Cr^{-3-\sigma+(p-1)(-2-\sigma)} \le Cr^{-1-p(2+\sigma)} <Cr^{-1-n}
\end{align}
and
\begin{align}\label{S203}
\lim_{r \to +\infty}\frac{\text{Vol}(\{r(x)=r\})}{nw_nr^{n-1}}=1
\end{align}
by our definitions of $r$ and AE manifolds. Moreover \eqref{E13} is true since $p >\frac{n}{2+\sigma} \ge 1$ by our assumption $\sigma \le n-2$.

So from \eqref{E13} there exists $c_2>0$ such that 

\begin{equation}\label{E14}
\int u^p \, dV \le c_2
\end{equation}
on any time slice.

For a fixed $x \in M^n$ and any $t \ge 1$ by using Harnack inequality Theorem \ref{T307} we have
\begin{align}\label{S203}
u^p(y,2t) \ge 2^{-c_1p/2}\exp(-p(1+t)/2t)u^p(x,t)
\end{align}
for any $y \in B_{g(t)}(x,(1+t)^{\frac{1}{2}-\epsilon})$.
Therefore,
\begin{align}\label{E15}
c_2 \ge &\int_M u^p(y,2t) \, dV_{g(2t)}(y) \ge \int_{B_{g(t)}(x,(1+t)^{\frac{1}{2}-\epsilon})} u^p(y,2t) \, dV_{g(2t)}(y) \notag \\
\ge & 2^{-c_1p/2}\exp(-p(1+t)/2t) \text{Vol}_{g(2t)}\left(B_{g(t)}(x,(1+t)^{\frac{1}{2}-\epsilon})\right) u^p(x,t) \notag \\
\ge & c_3 \text{Vol}_{g(2t)}\left(B_{g(t)}(x,(1+t)^{\frac{1}{2}-\epsilon})\right) u^p(x,t)
\end{align}
for some constant $c_3=2^{-c_1p/2}e^{-p} \le 2^{-c_1p/2}\exp(-p(1+t)/2t)$ for any $t \ge 1$.

The evolution equation for the volume of any compact set $K \subset M^n$ is
$$
\frac{d}{dt}\left( \int_K \,dV \right)=\int_K -R \, dV \ge \frac{-\epsilon}{1+t}\int_K \, dV.
$$
So we have 
\begin{equation}\label{E16}
\text{Vol}_{g(t)}(K) \ge (1+t)^{-\epsilon} \text{Vol}_{g(0)}(K).
\end{equation}

On the other hand, by the same reason 
\begin{equation}\label{E16_1}
d_{g(t)}(x,y)\le (1+t)^{\epsilon}d_{g(0)}(x,y)
\end{equation}
for any $x,y \in M^n$.

So from \eqref{E15} \eqref{E16} and \eqref{E16_1} we have
\begin{align}
c_2 \ge & c_3 \text{Vol}_{g(2t)}\left(B_{g(t)}(x,(1+t)^{\frac{1}{2}-\epsilon})\right) u^p(x,t) \notag \\
\ge & c_3 (1+2t)^{-\epsilon} \text{Vol}_{g(0)}\left(B_{g(t)}(x,(1+t)^{\frac{1}{2}-\epsilon})\right) u^p(x,t) \notag \\
\ge & c_3 (1+2t)^{-\epsilon} \text{Vol}_{g(0)}\left(B_{g(0)}(x,(1+t)^{\frac{1}{2}-2\epsilon})\right) u^p(x,t) \notag \\
\ge & c_4 (1+2t)^{-\epsilon} (1+t)^{(\frac{1}{2}-2\epsilon)n} u^p(x,t)
\end{align}
for some $c_4>0$ by the AE condition of $g(0)$.

Hence we have
\begin{equation} \label{E17}
u(x,t) \le C(1+t)^{\frac{\epsilon-(1/2-2\epsilon)n}{p}}.
\end{equation}

Then if $\epsilon$ is sufficiently small which depends on $p$ and $n$, then $\frac{\epsilon-(1/2-2\epsilon)n}{p} < -1$ and we can choose $\delta=-1-\frac{\epsilon-(1/2-2\epsilon)n}{p}>0$.

On the other hand if $t \le 1$ the conclusion is obvious since $u$ is uniformly bounded on compact time interval.
\end{proof}

With Theorem \ref{T308}, we prove the following estimate for the curvature operator.
\begin{thm} \label{T309}
$|\text{Rm}| \le \dfrac{C_0}{(1+t)^{1+\delta_0}}$
for some constants $C_0,\delta_0 >0$.
\end{thm}
\begin{proof}
Under Ricci flow, we have the following lemma by direct computations.
\begin{lem}\label{L_8}
Let T be a time-dependent tensor on $M$ and $u$ is a positive solution of $\partial_tu=\Delta u$, then
$$
(\partial_t-\Delta)\frac{|T|^2}{u^2}=\frac{2}{u}\nabla u \cdot \nabla\frac{|T|^2}{u^2}-2\frac{|u\nabla T-\nabla u T|^2}{u^4}+\frac{(\partial_t-\Delta)|T|^2}{u^2}.
$$
\end{lem}
Let $W=\dfrac{|\text{Rm}|^2}{u^2}$, then from the Lemma \ref{L_8} we have
\begin{align}
\partial_tW = & \Delta W+\frac{2}{u}\nabla u \cdot \nabla W-2\frac{|u\nabla \text{Rm}-\nabla u \text{Rm}|^2}{u^4}+P \notag \\
\le & \Delta W+\frac{2}{u}\nabla u \cdot \nabla W+P,
\end{align}
where
$$
P=\frac{8(B_{ijkl}+B_{ikjl})R_{ijkl}}{u^2} \quad \text{and} \quad B_{ijkl}=-R_{pijq}R_{qlkp}.
$$

We have the following estimate for $P$.
\begin{equation}\label{E19}
P\le \frac{16|\text{Rm}|^3}{u^2} \le \frac{16\epsilon}{1+t}W
\end{equation}
where the last inequality is from \eqref{E303}.

As in the proof of Theorem \ref{R01}, $\frac{2}{u}\nabla u$ is bounded on $M^n \times [0,T]$ for any $T>0$. From Theorem \ref{T103} we conclude that

\begin{equation}\label{E21}
W=\frac{|\text Rm|^2}{u^2}\le C (1+t)^{16 \epsilon}
\end{equation}
for some constant $C>0$.

Therefore, from Theorem \ref{T308} we know that 
there exists $C_0>0$ such that
\begin{equation} \label{E18}
|\text{Rm}| \le C_0 u(1+t)^{8 \epsilon} \le \frac{C_0}{(1+t)^{1+\delta-8\epsilon}}
\end{equation}
where we can take $\delta_0=\delta-8\epsilon >0$ by choosing $\epsilon$ to be small enough.
\end{proof}

Now from the proof of Theorem \ref{T308}, we know that for any $\sigma_0$ slightly smaller than $\sigma$, 
$$
u(x,t)\le Ct^{-1-\sigma_0/2}
$$
Therefore, $|\text{Rm}|\le C t^{-1-\sigma_0/2}$. In other words, we have shown $\delta_0$ can be chosen to be any number less than $\sigma /2$.

We have the following version of Shi's estimate, see also \cite{Shi89}, 
\begin{thm} \label{ST001}
For any $k=0,1,\ldots$ 
$$
|\nabla ^k \text{Rm}|\le  C_kt^{-1-\delta_0-k/2}.
$$
\end{thm}

\begin{proof}
From the Theorem \ref{T309} the conclusion is true for $k=0$. We assume by induction that it holds for any $0 \le l <k$.\\
For any fixed $s \ge 1$, we let 
$$
F(x,t)=(t-s)^k|\nabla^k \text{Rm}|^2+C_1(t-s)^{k-1}|\nabla^{k-1} \text{Rm}|^2 +\cdots+C_k|\text{Rm}|^2
$$
on $M \times [s,\infty)$.
From the evolution equation of $|\nabla^k \text{Rm}|^2$
\begin{align}
\partial_t|\nabla^k\text{Rm}|^2&=\Delta|\nabla^k\text{Rm}|^2-2|\nabla^{k+1}\text{Rm}|^2+\sum_{l=0}^k\nabla^l\text{Rm}*\nabla^{k-l}\text{Rm}*\nabla^k\text{Rm} \notag \\
& \le \Delta|\nabla^k\text{Rm}|^2-2|\nabla^{k+1}\text{Rm}|^2+C\sum_{l=0}^k|\nabla^l\text{Rm}||\nabla^{k-l}\text{Rm}||\nabla^k\text{Rm}|
\end{align}
we have by the induction,
$$
(t-s)^{k}|\nabla^l\text{Rm}||\nabla^{k-l}\text{Rm}||\nabla^k\text{Rm}| \le Ct^{-2-2\delta_0}(t-s)^{k/2}|\nabla^k \text{Rm}| \le Ct^{-2-2\delta_0}F^{1/2}
$$
for $0<l<k$ and 
$$
(t-s)^{k}|\nabla^l\text{Rm}||\nabla^{k-l}\text{Rm}||\nabla^k\text{Rm}|=(t-s)^{k}|\text{Rm}||\nabla^k\text{Rm}|^2 \le Ct^{-1-\delta_0}F
$$
for $l=0$ or $l=k$.

Therefore, we can find nonnegative constants $C_1,C_2,\ldots, C_k$, such that $F$ satisfies the following equation
\begin{equation}\label{E19}
\partial_t F \le \Delta F+Ct^{-2-2\delta_0}(F^{1/2}+t^{1+\delta_0}F).
\end{equation}

We consider the ODE
\begin{align}
\frac{d\phi}{dt}&=Ct^{-2-2\delta_0}(\phi^{1/2}+t^{1+\delta_0}\phi), \notag \\ 
\phi(s)&=\bar C s^{-2-2\delta_0}
\end{align}
where $\bar C=C_kC_0^2$.
Now $F(x,s) \le \phi(s)$ since $F(s)=C_k|\text{Rm}|^2 \le \bar Cs^{-2-2\delta_0}$.

Since $\phi(t)$ is increasing, $\phi(t) \ge \bar Cs^{-2-2\delta_0} \ge \bar Ct^{-2-2\delta_0}$ for $t \ge s$ and hence
\begin{align}
\frac{d\phi}{dt} =&Ct^{-2-2\delta_0}\phi^{1/2}+Ct^{-1-\delta_0}\phi \\ \notag
\le& Ct^{-1-\delta_0}\phi.
\end{align}
Then it is easy to show $\phi(t) \le Cs^{-2-2\delta_0}e^{Ct^{-\delta_0}}\le Cs^{-2-2\delta_0}$  for $t \ge s \ge 1$.

Now from Theorem \ref{T103}, we conclude that 
$$
F(2s) \le Cs^{-2-2\delta_0}.
$$
In other words,
$$
s^k|\nabla^k\text{Rm}|^2(2s) \le C_ks^{-2-2\delta_0}.
$$
Since $s$ is an arbitrary positive number, we have 
$$
|\nabla^k\text{Rm}|(t) \le Ct^{-1-\delta_0-k/2}
$$
which completes the induction process.
\end{proof}

Thus there exists a metric $g_{\infty}$ such that $g(t)$ converges to $g_{\infty}$ smoothly as $t \to \infty$. Moreover, argue as before
$$
\mu(g_{\infty},\tau) \ge \limsup_{t \to \infty} \mu(g(t),\tau) \ge  \limsup_{t \to \infty}\mu(g(0),\tau+t)=0
$$ for any $\tau >0$.

Then from Theorem \ref{T201} $(M^n,g_{\infty})=(\mathbb R^n,g_{E})$. In particular, $M^n$ is diffeomorphic to $\mathbb R^n$.

\section{Proof of Theorem \ref{T101}}

In this section, we prove our first main theorem.

We first recall the definition of weighted function space, see for example \cite{LP87}. Let $(M,g)$ be an AE manifold with the AE end $E$, the weighted space $C_{\beta}^k(E)$ consists of $C^k$ functions $u$ for which the norm 
$$
\|u\|_{C_{\beta}^k}=\sum_{i=0}^k \sup_M r^{-\beta+i}|\nabla^iu|
$$
is finite.  The weighted H\"older space $C^{k,\alpha}_{\beta}(E)$ is defined for $0<\alpha<1$ as the set of $u \in C_{\beta}^k(E)$ for which the norm
$$
\|u\|_{C^{k,\alpha}_{\beta}}=\|u\|_{C_{\beta}^k}+\sup_{x,y}\left(\min \{r(x),r(y)\}\right)^{-\beta+k+\alpha}\frac{|\nabla^ku(x)-\nabla^ku(y)|}{|x-y|^{\alpha}}
$$
is finite.

Then we have the following convergence result in the weighted space.

\begin{thm} \label{T_10}
For any $\sigma' \in \left(\frac{n-2}{2}, \sigma \right)$, we have $g_{ij}(t)$ converges to $g_{ij}(\infty)$ in $C_{-\sigma'}^{\infty}$ as $t \to \infty$. In particular, $(g_{ij}(\infty),E)$ is an AE coordinate system on $M^n$.
\end{thm}

\begin{proof}
We first prove a lemma 
\begin{lem} \label{L405}
There exist $C_k,\eta_k >0$ such that
$$
|\nabla^k \text{Rm}|(x,t)\le C_k t^{-1-\eta_k}r^{-k-\sigma'}, \quad k=0,1,\ldots
$$
for all $(x,t) \in M \times [0,\infty)$.
\end{lem}

\emph{Proof of the lemma}:
We choose $\sigma_1,\sigma_0$ such that $\sigma'<\sigma_1<\sigma_0 <\sigma$ and $\delta_0=\sigma_0/2$ in Theorem \ref{ST001}. 

We consider a domain $D_k=\{(x,t) \in M \times [0,\infty) \, |\, r(x) \ge t^{a_k}\}$ in the spacetime where $a_k>1/2$ to be determined later.

For $(x,t) \notin D_k$, from Theorem \ref{ST001}, we have 
\begin{equation} \label{E4071}
|\nabla^k \text{Rm}|\le C_kt^{-1-\sigma_0/2-k/2} \le C_kt^{-1-\eta_k}r^{-k-\sigma'}
\end{equation}
for some $\eta_k>0$ when $a_k$ is sufficiently close to $1/2$.

For $(x,t)\in D_k$, we have the following estimate.

\emph{Claim}: $|\nabla^k\text{Rm}|^2 \le Cr^{-4-2\sigma_1-2k}$ on $D_k$.

\emph{Proof of the claim}:
Let $h_k=r^{4+2\sigma_1+2k}$ and $w_k=h_k|\nabla^k\text{Rm}|^2$, from \eqref{E201a} we have
\begin{equation} \label{RE502}
(\partial_t-\Delta)w_k \le B_kw_k-2\nabla\log h_k\nabla w_k+C\sum_{l=0}^kh_k|\nabla^l\text{Rm}||\nabla^{k-l}\text{Rm}||\nabla^k\text{Rm}|
\end{equation}
where $B_k=\frac{2|\nabla h_k|^2-h_k\Delta h_k}{h_k^2}$ is uniformly bounded by $r^{-2} \le t^{-2a_k}$.

For $k=0$, we have
$$
(\partial_t-\Delta)w_0 \le -2\nabla\log h_0\nabla w_0+Ct^{-1-\delta_0'}w_0
$$
for some $\delta_0'=\min\{2a_0-1,\sigma_0/2\}>0$.

Moreover, on $\partial D_0$ we have
\begin{equation}
|\text{Rm}|\le Ct^{-1-\sigma_0/2} = C r^{-(1+\sigma_0/2)/a_0} \le Cr^{-2-\sigma_1}
\end{equation}
for $a_0$ sufficiently close to $1/2$.

Now we apply Theorem \ref{T103} on $D_k$ to conclude that the claim holds for $k=0$. Note that even though in Theorem \ref{T103} there is no boundary in spacetime for $t >0$, if we go through the proof, see \cite[Theorem $12.14$]{CCGGIIKLLN10}, the contradiction is derived at an interior point as long as the conclusion holds also on the boundary.

Now we assume that the claim holds for all $0 \le l <k$, then by induction on $D_k$ we have
$$
h_k|\nabla^l\text{Rm}||\nabla^{k-l}\text{Rm}||\nabla^k\text{Rm}| =h_k|\text{Rm}||\nabla^k\text{Rm}|^2 \le t^{-1-\sigma_0/2}w_k
$$
for $l=0$ or $l=k$ and 
$$
h_k|\nabla^l\text{Rm}||\nabla^{k-l}\text{Rm}||\nabla^k\text{Rm}| \le Cr^k|\nabla^k\text{Rm}|=Cr^{-\sigma_1-2}w_k^{1/2} \le Ct^{-a_k\sigma_1-2a_k}w_k^{1/2}
$$
for $0<l<k$.\\
Therefore from \eqref{RE502} we have
$$
(\partial_t-\Delta)w_k \le -2\nabla\log h_k\nabla w_k+Ct^{-1-\delta_k'}(w_k+w_k^{1/2})
$$
for some $\delta_k'>0$.

On the other hand, on $\partial D_k$ we have by Shi's estimate
\begin{equation}
|\nabla^k \text{Rm}|\le C_k t^{-k/2}t^{-1-\sigma_0/2} = C_k r^{-(1+k/2+\sigma_0/2)/a_k} \le C_k r^{-2-k-\sigma_1}
\end{equation}
when $a_k$ is chosen to be sufficiently close to $1/2$.

So from maximum principle, we conclude that $w_k$ is uniformly bounded on $D_k$ and the claim holds for $k$ as well.

Therefore, on $D_k$ we have
$$
|\nabla ^k \text{Rm}| \le C_k r^{-2-k-\sigma_1} \le C_k t^{-1-\eta_k}r^{-k-\sigma'}
$$
for some $\eta_k>0$ and $a_k$ close to $1/2$.

Thus the proof of lemma is complete.

With the same argument in Theorem \ref{R01}, we conclude that $g_{ij}(t)$ converges to $g_{ij}(\infty)$ in $C_{-\sigma'}^{\infty}$ because the term $t^{-1-\eta_k}$ guarantees that $|\nabla^k \text{Rm}|$ is integrable with respect to time at infinity. In other words, $g_{ij}(\infty)$ is an AE coordinate system with a smaller order $\sigma'$ for the Euclidean space.
\end{proof}
Now we continue to prove Theorem \ref{T101}. We choose a smooth function $\eta$ such that $\eta=0$ outside of the AE end $E$ and $\eta=1$ when $r$ is large.

Let $\chi(t)=(\partial_ig_{ij}(t)-\partial_jg_{ii}(t))\partial_j$ be a vector field on the AE end, by the definition of mass,
\begin{align} \label{E411}
m(g(t))=& \lim_{r \to \infty} \int_{S^r} \chi(t)\lrcorner \, dV_{g_E} \notag \\
=&  \lim_{r \to \infty} \int_{S^r} \eta\chi(t)\lrcorner \, dV_{g_E} \notag \\
=& \int \eta \text{div}(\chi(t))+\left <\chi(t),\nabla \eta \right> \, dV.
\end{align}

On the other hand, we have, see \cite[$(9.2)$]{LP87},
\begin{align}
R= &g^{jk}(\partial_i \Gamma_{jk}^i-\partial_k \Gamma_{ij}^i+\Gamma_{il}^i\Gamma_{jk}^l-\Gamma_{kl}^i\Gamma_{ij}^l) \notag \\
 =& \partial_j(\partial_ig_{ij}-\partial_j g_{ii})+E(g)
\end{align}
where $E(g)$ is some universal analytic expression that is polynomial in $g$, $\partial g$ and $\partial^2g$ such that $E=O(r^{-2\sigma'-2})$. Moreover,
$$
|E(g(t))-E(g(\infty))| \le C\|g(t)-g(\infty)\|_{C_{-\sigma'}^2}r^{-2\sigma'-2}.
$$

By taking the difference of equations of $R(t)$ and $R(\infty)=0$, we have
\begin{align}\label{E412a}
R(t)=& \partial_j(\partial_ig_{ij}(t)-\partial_j g_{ii}(t))- \partial_j(\partial_ig_{ij}(\infty)-\partial_j g_{ii}(\infty))+E(g(t))-E(g(\infty)) \\ \notag
=&\text{div}\chi(t)-\text{div}\chi(\infty)+E(g(t))-E(g(\infty))
\end{align}
and hence
\begin{align}\label{E412}
|\text{div}\chi(t)-\text{div}\chi(\infty)-R(t)| \le C\|g(t)-g(\infty)\|_{C_{-\sigma'}^2}r^{-2\sigma'-2}.
\end{align}
From \eqref{E411} and \eqref{E412} we have
\begin{align} \label{E413}
m(g(0))=& \lim_{t \to \infty} m(g(t))=\lim_{t \to \infty}m(g(t))-m(g(\infty)) \notag \\
=& \lim_{t \to \infty} \int \eta (\text{div}\chi(t)-\text{div}\chi(\infty))+\left <\chi(t)-\chi(\infty),\nabla \eta \right> \, dV_{g_E} \notag \\
\ge & \lim_{t \to \infty} \int \eta R(t)-C\eta \|g(t)-g(\infty)\|_{C_{-\sigma'}^2}r^{-2\sigma'-2}+\left <\chi(t)-\chi(\infty),\nabla \eta \right> \, dV_{g_E}.
\end{align}

Now since $\sigma'>\frac{n-2}{2}$, $\eta r^{-2\sigma'-2}$ is integrable. In addition, $\chi(t)-\chi(\infty)$ converges to $0$ on the support of $\nabla \eta$ and $\|g(t)-g(\infty)\|_{C_{-\sigma'}^2}$ tends to $0$, so we have from \eqref{E413},
\begin{equation} \label{E414}
m(g(0))\ge  \lim_{t \to \infty} \int \eta R(t) \, dV_{g_E} \ge 0.
\end{equation}

\begin{rem}
From the above proof, we have shown
\begin{equation} \label{E415}
m(g(0)) = \lim_{t \to \infty} \int R(t) \, dV_t
\end{equation}
since $g(t)$ converges to $g_E$ uniformly on any compact set.
\end{rem}

If the equality holds, we have by \eqref{E415} $\lim_{t \to \infty} \int R(t) \, dV_t=0$.

On the other hand 
\begin{align}
\frac{d}{dt}\left(\int R \, dV\right)= &\int \Delta R+2|Rc|^2-R^2 \, dV \notag \\
= &\int 2|Rc|^2-R^2 \, dV \notag \\
\ge & -\frac{n-2}{n}\int{R^2} \, dV \notag \quad (\text{from} \,\,|Rc|^2\ge \frac{R^2}{n})\\
\ge & -\frac{C}{(1+t)^{1+\delta}}\int R \, dV
\end{align}
where the second inequlity holds since $\lim_{r \to \infty} \int_{S_r}|\nabla R(t)|\,d\sigma=0$ and hence $\int \Delta R \, dV=0$. The last inequality follows from Theorem \ref{T309}.

Taking the integration on both sides, $\lim_{t \to \infty} \int R(t) \, dV_t$ cannot be $0$ unless $R(t)\equiv 0$, which is a contradiction by our original assumptions. In other words, the only possibility for $m(g(0))=0$ is when $(M^n,g)=(\mathbb R^n, g_E)$.

Thus, we have completed the proof of Theorem \ref{T101}.

\section {Ricci flow with surgery on AE manifold}
In this section, we define the Ricci flow with surgery on an AE manifold. Most definitions and notations are from \cite{P031} \cite{MT07} \cite{BBM11} and \cite{KL08} with slight modifications. We assume from now on $M$ is an orientable Riemannian AE $3$-manifold with $R > 0$ unless otherwise specified.

First of all we fix a surgery model, see \cite[Section $2$]{P031} and \cite[Chapter $12$]{MT07},
\begin{defn}(surgery model)\label{E501}
Consider $M_{\text{stan}}=\mathbb R^3$ with its natural $SO(3)$-action, then there is a complete metric $g_{\text{stan}}$ on $M_{\text{stan}}$ such that
\begin{enumerate}

\item $g_{\text{stan}}$ is $SO(3)$-invariant.

\item $g_{\text{stan}}$ has nonnegative sectional curvature.

\item There is a compact ball $B \subset M_{\text{stan}}$ so that the restriction of the metric $g_{\text{stan}}$ to the complement of this ball is isometric to the product $(S^2, h) \times (\mathbb R^+,ds^2)$ where $h$ is the round metric of scalar curvature $1$ on $S^2$.

\item There is a standard Ricci flow $(M_{\text{stan}},g_{\text{stan}}(t)), 0\le t <1$ such that $1$ is the singular time.

\end{enumerate}
\end{defn}

For an AE manifold $M^3$, under Ricci flow, we either have long time existence or the metric goes singular at some finite time. In the latter case, we modify the resulting limit by surgery, which cuts off high curvature parts and add standard capped tubes, so as to produce a new manfiold with an AE end which serves a new initial condition for Ricci flow. 
Now we clarify the process of surgery at the first singular time for example. 

Let $(M,g(t)), 0 \le t <T$ be the Ricci flow solution where $T$ is the first singular time. Let $\Omega \subset M$ be a subset defined by 
$$
\Omega=\{x\in M | \limsup_{t \to T} R_g(x,t) < \infty\}.
$$
Then we have the following properties:
\begin{thm} \label {T501}
\begin{enumerate}

\item As $t \to T$ the metric $g(t)|_{\Omega}$ limit to $g(T)$ uniformly in the $C^{\infty}$-topology on every compact sets of $\Omega$.

\item Every end of a connected component of $\Omega$ is contained in a strong $\epsilon$-tube.

\item There exists $r>0$ such that any $x\in \Omega \times \{T\}$ with $R(x) \ge r^{-2}$ has a strong $(C,\epsilon)$-canonical neighborhood in $\widehat{\mathcal M}=M \times [0,T) \cup_{\Omega\times [0,T)}(\Omega\times [0,T])$.

\item There exists a compact set $K \subset M$ such that $|\text{Rm}|$ is bounded on $K^c \times [0,T)$. In particular, $K^c \subset \Omega$.

\item The scalar curvature $R(g(T))$ is a proper function from $\Omega \to (0,\infty)$.
\end{enumerate}
\end{thm}
\begin{proof}
The proof of $1-3$ can be found in \cite[Theorem $11.19$]{MT07}. $4$ is proved by pseudolocality, see \cite[Theorem $1.1$]{CTY11}.
To prove $5$, we need the following lemma.

\begin{lem} \label{L501}
There exists a compact set $K$ such that $g(T)$ has an AE coordinate system on $K^c=M-K$.
\end{lem}
\emph{Proof of the lemma}: From \cite[Theorem $1.1$]{CTY11}, there exist a compact set $K$ and $S>0$ such that $|\text{Rm}(x,t)|\le S$ on $K^c\times [0,T)$.
Enlarge $K$ if necessary, we can assume $g_{ij}(0)$ is an AE coordinate system on $K^c$ and $\partial K$ is smooth. Then we can use the same argument in Theorem \ref{R01} on the parabolic cylinder $K^c \times [0,T)$ to conclude that  $g(T)$ has an AE coordinate system on $K^c$.

Let $\{x_n\}_{n\in \mathbb N}$ be a sequence in $\Omega$ such that $0< c \le R(x_n,T)\le C $ for some constants $0<c<C$. Since by the Lemma \ref{L501}, $g(T)$ has curvature bounded by $Cr^{-2-\sigma}$, all $x_n$ are contained in a compact set of $M$. Then we assume, by taking a subsequence if necessary, $x_n$ converges to a point $x_{\infty}$ in $M$. If $x_{\infty}$ is not in $\Omega$, by Lemma \ref{L601} in the next section, we have $R(x_n,T)$ goes to infinity which is a contradiction.

Thus, the proof of Theorem \ref{T501} is complete.
\end{proof}

\begin{rem}
We call $K^c$ in Lemma \ref{L501} the AE end of $\Omega$.
\end{rem}

We fix $0<\rho<r$ where $r$ is the constant from Theorem \ref{T501}$(3)$ and define $\Omega_{\rho} \subset \Omega$ be the closed subset of all $x\in \Omega$ for which $R(x,T)\le \rho^{-2}$. For a component $\Omega_1$ of $\Omega$ which contains no point of $\Omega_{\rho}$, by the canonical neighborhood theorem, one of the following holds, see \cite[Lemma $11.28$]{MT07}:
\begin{enumerate}

\item $\Omega_1$ is a strong double $\epsilon$-horn and is dffeomorphic to $S^2 \times \mathbb R$.

\item $\Omega_1$ is a $C$-capped $\epsilon$-horn and is diffeomorphic to $\mathbb R^3$ or a punctured $\mathbb RP^3$.

\item $\Omega_1$ is a compact component and is diffeomorphic to $S^3/\Gamma$, $S^1 \times S^2$ or $\mathbb RP^3\# \mathbb RP^3$.
\end{enumerate}
Those are all possibilities if $M$ is orientable.

Let $\Omega^0(\rho)$ be the union of all components of $\Omega$ containing points of $\Omega_{\rho}$, then $\Omega^0(\rho)$ has finitely many components and is a union of the AE end and finitely many strong $\epsilon$-horns each of which is disjoint from $\Omega_{\rho}$. The finiteness of horns can be derived from the properness of $R(T) \to (0,\infty)$ and the rest arguments can be found in \cite[Lemma $11.30$]{MT07}.

Next, we have the following lemma which asserts the existence of a strong $\delta$-necks on which we will do surgeries.

\begin{lem}\cite[Theorem $5.1$]{BBM11} \label{L502}
For any $\delta>0$, there exist $h\in (0,\delta \rho)$ and a constant $D=D(\delta,\rho)$ such that the following holds: Let $x,y,z \in \Omega$ such that $R(x,t)\le \rho^{-2}, R(y,t)=h^{-2}$ and $R(z,t)\ge Dh^{-2}$. Assume that there is a curve $\gamma$ in $\Omega^c_\rho$ connecting $x$ to $z$ via $y$. Then $(y,t)$ is center of a strong $\delta$-neck.
\end{lem}

Now for the surgery parameters $r,\delta<1$ we set $\rho=r\delta$, then the scale $h=h(\rho,r)=h(\delta,r)$ and $D=D(\rho,r)=D(\delta,r)$ are determined. Moreover,  we require that
\begin{align}\label{EK500}
\lim_{\delta \to 0}\frac{D(\delta, r)h(\delta,r)}{\rho^3}=0
\end{align}
since the proof of lemma \ref{L502} argues by contradiction by choosing two independent sequences $h_i \to 0$ and $D_i \to +\infty$.

We say $(M_+,g_+)$ is obtained from $(\Omega, g(T))$ by $(r,\delta)$-surgery at time $T$ if 
\begin{enumerate}

\item $M_+$ is obtained from $\Omega$ by removing components disjoint from $\Omega_{\rho}$ and cutting along a locally finite collection of disjoint $2$-spheres, capping off $3$-balls.

\item All $x \in M_+\backslash M(T)$ are contained in a surgery cap and the cutting and capping are done on a strong $\delta$-neck centered at a point $y$ with $R(t,T)=h^{-2}$.

\item $(M_+,g_+)$ is pinched toward positive curvature.
\end{enumerate}

Now we show $(r,\delta)$-surgery must exist, see \cite[Lemma $7.6$]{BBM11}. 

By Zorn's lemma, on $\Omega$ there exists a maximal collection $\{N_i\}$ of pairwise disjoint $\delta$-necks centered at $y_i$ with $R(y_i,T)=h^{-2}$. Then from lemma \eqref{L502}, every components of $\Omega \backslash \cup_i N_i$ has the scalar curvature either less than $Dh^{-2}$ or greater than $\rho^{-2}$. Then we remove all the components of the second kind and do surgeries on those $\delta$-necks $N_i$.

Now we let $M_+$ be the resulting manifold and $R(g_+)\in (0,Dh^{-2}]$. From the construction we know that each component of $M_+$ contains at least one point $p$ at which $R(p,T)\le \rho^{-2}$, hence there are at most finitely many components by the properness of $R$. Moreover one of the component $M_+^0$ containing the AE end of $M_+$ is an AE manifold with the same order $\sigma$ as $M$. In addition, the mass of $(M_+^0, g(T))$ is well defined and is equal to that of $M$, by the same argument in \cite{DM07}.

In general, we can construct three weakly decreasing parameter functions $r(t),\delta(t), \kappa(t), t\in [0, \infty)$ to regulate the surgery process such that $r(t)$ is a canonical neighnorhood scale function. The following existence theorem is proved in \cite[Theorem $15.9$]{MT07}, see also \cite[Theorem $1.2$]{BBM11}.

\begin{thm} \label{T502}
There exists a Ricci flow with surgery $(\mathcal M,g_{\mathcal M})$ on $[0, \infty)$ with the initial condition $(M,g)$ and decreasing functions $\delta(t),r(t),\kappa(t):[0,\infty) \to \mathbb R^+$ such that the following holds,
\begin{enumerate}
\item $(\mathcal M,g_{\mathcal M})$ has curvature pinched toward positive;
\item the flow satisfies the strong $(C,\epsilon)$-canonical neighborhood theorem with parameter $r(t)$ on $[0,\infty)$;
\item the flow is $k(t)$-noncollapsed on $[0,\infty)$ on scales $\le \epsilon$ and
\item for any singular time $t$ the surgery is performed with control $\delta(t)$ at scale $h(t)=h(\rho(t),\delta(t))=h(r(t)\delta(t),\delta(t))$.
\end{enumerate}
\end{thm} 

Next we show that surgery times do not accumulate.
\begin{thm} \label{T503}
Let $(\mathcal M,G)$ be a Ricci flow with surgery on $[0, \infty)$ with the initial condition $(M,g)$ with parameter functions  $\delta(t),r(t),\kappa(t)$, we show that on each compact interval $I$ of $[0,\infty)$, we have at most finitely many surgeries.
\end{thm}

\begin{proof}
Since all the parameter functions are decreasing, we can choose uniform parameters $\delta, r$ and $\kappa$ on $I$. Therefore functions $h$ and $D$ are uniformly determined as well. At each singular time $t$, by our construction $R(x,t) \le Dh^{-2}$. Since curvature is pinched toward positive curvature, we can assume $|\text{Rm}|\le CDh^{-2}$. Now from the evolution equation of $|\text{Rm}|^2$ 
$$
\partial_t|\text{Rm}|^2 \le \Delta |\text{Rm}|^2+16|\text{Rm}|^3
$$
the regular Ricci flow exists at least for time $\frac{h^2}{16CD}$ from $t$. Since all constants are uniformly chosen, there are at most finitely many surgeries performed on $I$.
\end{proof}

\begin{rem}
Theorem \ref{T503} holds for all Ricci flows with surgery with normalized initial condition, which is satisfied after a scaling, if necessary, for our original manifold $M$.
\end{rem}

From the construction of Ricci flow with surgery, each time slice $(\mathcal M(t),g(t))$ consists of an AE manifold and a finite number of compact components. Moreover, we can recover the topology of $\mathcal M(0)=M$ by performing connected sum operations among $\mathcal M(t)$ and finitely many $S^3/\Gamma$ and $S^1 \times S^2$ for any $t>0$.

\section{Proof of Theorem \ref{T102}}

We first introduce the following definition.
\begin{defn}
For a Ricci flow with surgery $\mathcal M$, a connected open subset $\mathcal X \subset \mathcal M$ is called a path of components if for every time $t$, the intersection $\mathcal X(t)$ of $\mathcal X$ with each time-slice $\mathcal M(t)$ is a connected component of $\mathcal M(t)$.
\end{defn}
We set $\mathcal M_0$ to be the path of components of $\mathcal M$ such that $\mathcal M_0(t)$ is an AE manifold for any $t \ge 0$. 

Next we quote a local regularity lemma.

\begin{lem} \cite[Lemma $3.1$]{KL14} \label{L601}
Let $\mathcal M$ be a Ricci flow with surgery, with normalized intial condition. Given $T >\frac{1}{100}$, there are numbers $\mu=\mu(T) \in (0,1), \sigma=\sigma(T) \in (0,1), i_0=i_0(T)>0$ and $A_k=A_k(T)<\infty,k \ge 0$, with the following property. If $t \in (\frac{1}{100},T]$ and $|R(x,t)| < \mu\rho(0)^{-2}-r(T)^{-2}$, put $Q=|R(x,t)|+r(t)^{-2}$. Then
\begin{enumerate}
\item The forward/backward parabolic ball $P_{\pm}(x,t,\sigma Q^{-\frac{1}{2}})$ is unscathed, that is, with no intersection with the surgery cap.
\item $|\text{Rm}| \le A_0 Q, \text{inj} \ge i_0Q^{-\frac{1}{2}}$ and $|\nabla^k \text{Rm}|\le A_k Q^{1+\frac{k}{2}}$ on the union  $P_+(x,t,\sigma Q^{-\frac{1}{2}}) \cup P_-(x,t,\sigma Q^{-\frac{1}{2}})$ of the forward and backward parabolic balls.
\end{enumerate}
\end{lem} 

Now we consider a sequence of $\{\mathcal M^i\}$ of Ricci flows with surgery, where we let $\delta_i(0) \to 0$, hence $\rho_i$ and $h_i$ also go to $0$. We first prove a stability result, which shows that on the finite time interval, all surgeries are done in a compact set.
\begin{thm} \label{T601}
Let $\{\mathcal M^i\}$ be a sequence of Ricci flows with surgery with $\mathcal M^i(0)=M$ and $\lim_{i \to \infty} \delta_i(0)=0$. For any $S>0, T>0$, there exists a compact set $K \subset M$ such that for sufficiently large $i$, the cylinder $K^c \times [0,T]$ exists in $\mathcal M^i$ and $|\text{Rm}_i| \le S$.
\end{thm}
\begin{proof}
We prove it by contradiction.

Assume there is a sequence $\{x_j\}_{j\in \mathbb N}$ on $M$ with $d_g(x_j,\star)=2r_i$ where $\star$ is a fixed point on $M$ and $r_i \to \infty$ such that $|\text{Rm}_j|(x_j,t_j)>S$ for some $t_j \in [0,T]$.

By the AE condition, balls $(B_g(x_j,r_j),g,x_j)$ converges smoothly to $(\mathbb R^n, g_E, 0)$. Then there exists a $\theta>0$ sufficiently small such that  $B_g(x_j,r_j) \times [0,\theta]$ exists in $\mathcal M^j$ and for any $A>0$, restriction of $g_j$ on $B_g(x_j,A) \times [0,\theta]$ converges smoothly to the Euclidean metric on $B_{g_E}(0,A)\times [0,\theta]$.

Therefore for any $A>0$, we assume $|\text{Rm}| \le S/2$ on  $B_g(x_j,A)\times [0,\theta]$ for $j$ sufficiently large. From Lemma \ref{L601}, there exists $Q, \sigma, A_k,\theta'=\sigma Q^{-\frac{1}{2}}$, all of which depend on $S,T,r,\kappa,(M,g)$, such that the forward parabolic ball $P_+(x_j,\theta,\theta')$ and the backward parabolic ball $P_-(x_j,\theta,\theta')$ are unscathed and $|\nabla^k \text{Rm}|\le A_k Q^{1+\frac{k}{2}}$ with $\text{inj} \ge i_0Q^{-\frac{1}{2}}$ on $B_g(x_j,A)\times [\theta-\theta',\theta+\theta']$ for $j$ sufficiently large. By taking a diagonal subsequence, we have $B_g(x_j,r_j) \times [0,\theta+\theta']$ converges smoothly to the Euclidean metric on $\mathbb R^n \times [0,\theta+\theta']$.

Now we can continue this process, since $\theta'$ does not depend on $\delta^j$, to conlude that $B_g(x_j,r_i) \times [0,T]$ converges smoothly to the Euclidean metric on $\mathbb R^n \times [0,T]$ and $|\text{Rm}_j| \le S/2$ on  $B_g(x_j,1)\times [0,T]$. This is a contradiction.
\end{proof}

From Theorem \ref{T201}, we can find a constant $\epsilon_0>0$ such that $\mu_{S^2 \times \mathbb R}(g_c,1)\le -2\epsilon_0$, where $g_c$ is the standard metric on the cylinder with scalar curvature $R=1$. Therefore, we choose the parameter $\epsilon$ for the surgery as follows, for any $\epsilon$-neck with metric $g$ and center $p$, we have $\mu_{S^2 \times (-\epsilon^{-1},\epsilon^{-1})}(R(p)g,1)\le -\epsilon_0$.

Let $\mathcal M$ be a Ricci flow with surgery such that $r,\rho,h$ and $\delta$ are uniform surgery parameters. If $T$ is a surgery time, we consider the change of the $\mu$-functional from $\left(\mathcal M(T),g(T)\right)$ to $\left(\mathcal M(T^-),g(T^-)\right)$. Henceforth, we assume that $\left(\mathcal M(T^-),g(T^-)\right)$ and $\left(\mathcal M(T),g(T)\right)$ are pre-surgery and post-surgery Riemannian manifolds, respectively.

Now for a Riemannian manifold $(M,g)$, we have the following definition,

\begin{defn}\cite[(2-11)]{Z07}
$$
\lambda_{\sigma^2}(g)=\inf \left\{ \int \left(\sigma^2(4|\nabla v|^2+Rv^2)-v^2\log v^2\right)\, dV-n\log \sigma\, | \, v\in C^{\infty}(M), \,\|v\|_2=1         \right\}.
$$
\end{defn}
By our definition of $\overline{\mathcal W}(g,u,\tau)$ in \eqref{E202}, it is straightforward to compute, by setting $u=v(4\pi\sigma^2)^{\frac{n}{4}}$ that
\begin{align}
\mu(g,\sigma^2)=\lambda_{\sigma^2}(g)-n-\frac{n}{2}\log{4\pi}.
\end{align} 
In other words, $\mu(g,\sigma^2)$ and $\lambda_{\sigma^2}(g)$ are different by a constant.

If we set $g_1=\sigma^{-2}g$ and let $u_1$ be a minimizer of $\lambda_1(g_1)$, then we have, see \cite[(2-12)]{Z07}
\begin{align}\label{S100}
4\Delta_1u_1-R_1u_1+2u_1\log u_1+\Lambda u_1=0
\end{align} 
where $\Lambda=\lambda_1(g_1)$.

Now from  \cite[(2-13)]{Z07} we have
\begin{align}\label{S102}
\lambda_{\sigma^2}(g(T^-)) \le \Lambda+ck\left(1+\frac{4c\sigma^2}{h^2}\right)\frac{\int_U u_1^2 \,dV_{g_1}}{1-\int_U u_1^2 \,dV_{g_1}}.
\end{align} 
where $k$ is the number of surgery caps with scale $h$ and $U$ is any surgery cap.

To estimate the term $\int_U u_1^2 \,dV_{g_1}$, we have the following two lemmas, see \cite[Lemma $2.2$, $2.3$]{Z07}.
\begin{lem} \label{SL101}
$$
\sup_{\Omega^c_{\rho}}u_1^2\le c \max\left \{\left(\frac{\rho}{\sigma}\right)^{-3},1\right\}.
$$
\end{lem}

\begin{lem} \label{SL102}
Let $u$ be a positive solution to the inequality
$$
4\Delta u-Ru+2u\log u+\Lambda u \ge 0.
$$
Given a nonnegative function $\phi \in C^{\infty}(M)$ with $\phi \le 1$, suppose there is a smooth function $f$ that, when $R \ge 0$ in the support of $\phi$, satisfies
$$
4|\nabla f|^2 \le R-2\log^+u-3|\Lambda|/2 \quad \text{in the support of}\, \phi.
$$
Then
$$
\frac{1}{2}|\Lambda| \|e^f\phi u\|^2_2 \le 8 \sup_{x \in \text{supp}\nabla \phi} \left(e^{2f}(R-2 \log^+u-3|\Lambda|/2)+\|e^f\nabla \phi\|_{\infty}^2 \right) \|u\|^2_2.
$$
\end{lem}

Note that our Lemma \ref{SL102} is slightly different than Lemma $2.3$ in \cite{Z07} as we do not assume $\Lambda \le 0$. Since we impose a stronger restriction on $4|\nabla f|^2$, the proof is identical.

Now we fix a constant $\Lambda_0=n+\dfrac{n}{2}\log{4\pi}-\epsilon_0/2$. It is from Lemma \ref{SL101}, \ref{SL102} and the proof of \cite[Theorem $1.6$]{Z07} that there exists a small constant $\epsilon_1>0$ such that if $\dfrac{\rho}{\sigma} \le \epsilon_1$, then either $\lambda_{\sigma^2}(g(T^-)) \ge \Lambda_0$ or 
\begin{align}\label{S104}
\lambda_{\sigma^2}(g(T^-)) \le \lambda_{\sigma^2}(g(T))+ck(\sigma+1)^3h^3.
\end{align} 

Here the condition of $\dfrac{\rho}{\sigma} \le \epsilon_1$ is assumed to guarantee, see \cite[(2-14)]{Z07}, that
\begin{align}\label{S103}
\frac{R_1(x)}{2}\le R_1(x)-2\log^+ u_1(x)-3\Lambda_0/2 \le R_1(x)
\end{align}
on $\Omega^c_{\rho}$.

In terms of $\mu$-functional, it shows that if $\dfrac{\rho}{\sigma} \le \epsilon_1$, then either $\mu(g(T),\sigma^2) \ge -\epsilon_0/2$ or 
\begin{align}\label{S105}
\mu(g(T^-),\sigma^2) \le \mu(g(T),\sigma^2)+ck(\sigma+1)^3h^3.
\end{align} 

Now we take a sequence of Ricci flow with surgery $\{\mathcal M^i\}$ with a fixed AE manifold $(M,g)$ as the initial condition subject to a uniform $r(t)>0$ and surgery parameter function $\delta_i(0) \to 0$. 

From Theorem \ref{T202}, there exits a constant $T>0$ such that 
\begin{equation}\label{EK700}
\mu_M(g,\tau) \ge -\epsilon_0/2
\end{equation}
for any $\tau \ge T$. 

Then from Theorem \ref{T601}, there exists a compact set $K \subset M$ such that $|\text{Rm}_i| \le 1$ on $(M \backslash K) \times [0,T]$ and we can find a common AE coordinate system for all $g_i(T)$. On the other hand from maximum principle it is easy to show that $\mathcal M^i_0(T)\backslash (M-K)$ have uniform positive lower bound of scalar curvatures. Hence, from Theorem \ref{TC202} there exists $T'>T$ such that
\begin{equation}\label{EK7001}
\mu_M(g_i(T),\tau) \ge -\epsilon_0/2
\end{equation}
for any $\tau \ge T'-T$ and $i$.

Now since all $r(t)$ and $\delta_i(t)$ are decreasing, we can choose $r>0$, $\delta_i \to 0$ as constant parameters on the time interval $[0,T']$.

With all those preparations, Theorem \ref{T102} follows immediately from Theorem \ref{T101} and the following theorem.
\begin{thm} \label{T701}
There are finitely many surgeries for $\mathcal M^i_0$ for $i$ sufficiently large.
\end{thm}
\begin{proof}
Suppose the conclusion is false. Then we can assume for all $i$, $\mathcal M^i_0$ has infinitely many surgeries. In particular, we denote the first surgery time past $T$ by $T^i_{k_i}$ for $\mathcal M_0^i$ and all previous surgery times by $\{T^i_1,T^i_2,\cdots, T^i_{k_i-1} \}$. We also set $(\sigma^i_j)^2=T^i_{k_i}-T^i_{k_i-j}$ for $1 \le j \le k_i$ and $T^i_0=0$.

If $T^i_{k_i} \ge T'$, as $T^i_{k_i}$ is a singular time, we can find a sequence of points $\{p^i_v=(x^i_v,t^i_v)\}_{v \in \mathbb N}$ in $\mathcal M^i_0$ such that $t^i_v \to T^i_{k_i}$ and if $Q^i_j=R(x^i_v,t^i_v)$, $(\mathcal M^i_0(t^i_v),Q_v^ig(t^i_v),x^i_v)$ converges smoothly as $v \to \infty$ to a standard cylinder $(S^2 \times \mathbb R,g_c)$. Then we have
\begin{align}
-2\epsilon_0 \ge& \mu_{S^2 \times \mathbb R}(g_c,1) \notag\\ 
\ge&\lim_{v \to \infty}\mu(Q^i_vg_i(t^i_v),1) \notag\\
=&\lim_{v \to \infty}\mu(g_i(t^i_v)), 1/Q^i_v) \notag \\
\ge& \lim_{v \to \infty}\mu(g_i(T), 1/Q^i_v+t^i_v-T) \notag \\
=&\mu(g_i(T),T^i_{k_i}-T)
\end{align}
which contradicts \eqref{EK7001} since $T^i_{k_i}-T \ge T'-T$.

Therefore, we can assume all $T^i_{k_i} \le T'$.

By the same point-picking method as above, we have 
\begin{equation}\label{S106}
\mu(g(T^i_{k_i-1}),(\sigma^i_1)^2)=\mu(g(T^i_{k_i-1}),T^i_{k_i}-T^i_{k_i-1}) \le -2\epsilon_0
\end{equation}

We assume that $s$ is the largest integer among $1$ to $k_i$ such that $(\sigma^i_s)^2 < r^2$, where $r$ is the canonical neighborhood scale. As $T$ is a large number and $r$ is small, $T^i_{k_i-s}$ is a singular time. Now we can find a point $p$ which is the center of an $\epsilon$-neck such that $R(p)=(\sigma^i_s)^{-2}$. By our choice of $\epsilon$, we have $\mu(g(T^{i-}_{k_i-s}),(\sigma^i_s)^2) \le -\epsilon_0$. By using the monotonicity formula,
\begin{equation}\label{EK708}
\mu(g(T_{k_i-(s+1)}^i),(\sigma_{s+1}^i)^2) \le -\epsilon_0.
\end{equation}

Now let $l$ be the largest integer from $s+1$ to $k_i$ such that
\begin{equation}\label{S107}
\mu(g(T^i_{k_i-j}),(\sigma^i_j)^2) \le -2\epsilon_0/3.
\end{equation}
for any $s+1 \le j \le l$.

If $l \ne k_i$, then $T^i_{k_i-j}$ is a surgery time for any $j \in [s+1,l]$. Recall that by our assumption $(\sigma^i_j)^2 \ge r^2$. In this case, from \eqref{S105} ,\eqref{S106} and \eqref{S107} we have
\begin{equation}\label{EK701a}
\mu(g(T_{k_i-j}^{i-}),(\sigma_j^i)^2) \le \mu(g(T^i_{k_i-j}),(\sigma_j^i)^2)+ck(\sigma_j^i+1)^3h_i^3 \le \mu(g(T^i_{k_i-j}),(\sigma_j^i)^2)+Ckh_i^3
\end{equation}
since in this case $\dfrac{\rho_i}{\sigma_j^i} \le \dfrac{\rho_i}{r}= \delta_i \le \epsilon_1$ if $i$ is sufficiently large and $(\sigma_j^i)^2 \le T'$.

Now we estimate $k$. On $\mathcal M^i_0(T^{i-}_{k_i-j})$, we can find $k$ disjoint $\epsilon$-tubes and each contains an $\epsilon$-neck with center $p$ and $R(p)=\rho_i^{-2}$. The total volume of all $k$ tubes are at least $ck\rho_i^3$. Since all surgeries are done in a compact set $K$ whose volume is decreasing along the flow, we have 
\begin{equation}\label{EK702}
k \le C\rho_i^{-3}.
\end{equation}

Combining \eqref{EK701a} and \eqref{EK702}, we have
\begin{equation}\label{EK704}
\mu(g(T_{k_i-j}^{i-},(\sigma_j^i)^2) \le \mu(g(T^i_{k_i-j}),(\sigma_j^i)^2)+C\frac{h_i^3}{\rho_i^3}.
\end{equation}

Now we take sum from $s+1$ to $l$, then
\begin{equation}\label{EK706}
\mu(g(T_{k_i-l}^{i-},(\sigma_l^i)^2) \le -\epsilon_0+Ck_i\frac{h_i^3}{\rho_i^3}.
\end{equation}

We know that from Theorem \ref{T503}, the gap of two consecutive surgeries is at least $CD_i^{-1}h_i^2$, then
\begin{equation}\label{EK707}
k_i\le CD_iT'h_i^{-2}.
\end{equation}

Hence from \eqref{EK706}, 
\begin{equation}\label{EK708}
\mu(g(T_{k_i-l}^{i-}),(\sigma_l^i)^2) \le -\epsilon_0+CT'\frac{D_ih_i}{\rho_i^3}.
\end{equation}

From our choice of parameters, i.e. \eqref{EK500}, $\lim_{i \to \infty}\frac{D_ih_i}{\rho_i^3}=0$, so for $i$ sufficiently large, $CT'\frac{D_ih_i}{\rho_i^3} \le \epsilon_0/3$.

Therefore we have $\mu(g(T_{k_i-l}^{i-}),(\sigma_l^i)^2) \le -2\epsilon_0/3$. Again by using the monotonicity formula,
\begin{equation}
\mu(g(T_{k_i-(l+1)}^{i}),(\sigma_{l+1}^i)^2) \le -2\epsilon_0/3.
\end{equation}
But this contradicts the maximality of $l$.

Hence $l$ must be $k_i$ and in this case
\begin{equation}
\mu(g(0),(\sigma_k^i)^2) \le -2\epsilon_0/3.
\end{equation}

But this contradicts \eqref{EK700} since $(\sigma^i_{k_i})^2 \ge T$.

Thus, the proof of Theorem \ref{T701} is complete.
\end{proof}

\emph{Proof of Theorem \ref{T102}:} 
From Theorem \ref{T701} there exists a Ricci flow with surgery from $(M,g)$ such that there are only finitely many surgeries. Since the mass is preserved along Ricci flow and surgery times, $m(g)$ is nonnegative by Theorem \ref{T101}. If the equality holds, from Theorem \ref{T101} there is no surgery and $(M,g)=(\mathbb R^n,g_E)$.

\begin{cor}\cite[Corollary $6$]{RH12} \label{C701}
Any orientable AE $3$-manifold $M$ with scalar curvature $R \ge 0$ has the following diffeomorphism type
$$
M \cong \mathbb R^3 \# S^3/\Gamma_1 \#\ldots \# S^3/\Gamma_k \# (S^2 \times S^1) \# \ldots \# (S^2 \times S^1)
$$
where there are finitely many connected sums.
\end{cor}

\begin{proof}
From Theorem \ref{T701}, we have a Ricci flow with surgery $\mathcal M$ such that there are only finitely many surgeries on $\mathcal M_0$. After a large time $T$, the Ricci flow on $\mathcal M_0(T)$ has longtime existence, each of whose timeslice by Theorem \ref{T101} is diffeomorphic to $\mathbb R^3$. Moreover, at time $T$, all other finitely many components of $\mathcal M(T)$ are compact manifolds with $R>0$. Therefore they must extinct after finite time. Therefore we can recover the diffeomorphism type of $M$ by performing connected sum of $\mathbb R^3$ with finitely many $S^3/\Gamma$ and $S^2\times S^1$.
\end{proof}

\begin{rem}
Robert Haslhofer obtained the same result, see details in \cite[Corollary $6$]{RH12}, by using the min-max argument of Colding-Minicozzi \cite{CM08}.
\end{rem}

A natural question is whether we have the same result if we only assume $g_{ij}-\delta_{ij} \in C^2_{-\sigma}$.

\appendix
\newpage
\begin{appendices}
\section{Gradient Ricci solitons on ALE manifolds}
In this section we prove some results about Ricci gradient solitons on ALE manfolds.

\begin{defn} \label{DA1}
A smooth Riemannian manifold $(M^n,g)$ is called an asymptotically locally Euclidean (ALE) end of order $\sigma > 0$ if there exist a finite subgroup $\Gamma \subset O(n)$ acting freely on $\mathbb{R}^n \backslash B(0,R)$, a compact set $K \subset M^n$ and a $C^{\infty}$ diffeomorphism $\Phi: M^n \backslash K \to (\mathbb{R}^n \backslash B(0,A))/\Gamma$ such that under this identification,
\begin{align} 
g_{ij} &=\delta_{ij}+O(r^{-\sigma}),\\
\partial ^{|k|}g_{ij} &=O(r^{-\sigma-k}),
\end{align} 
for any partial derivatives of order $k$ as $r \to \infty$, where $r$ is the Euclidean distance. A complete, noncompact manifold $(M^n,g)$ is called ALE if $M^n$ can be written as the disjoint union of a compact set and finitely many ALE ends \cite{BKN89} \cite{TV05}. For an ALE end, if the group $\Gamma$ in the definition is trivial, we call it a \emph{trivial end} or AE end, otherwise we call it a \emph{nontrivial end}. As before, we assume that $r$ is a positive function defined on entire manifold $M^n$.
\end{defn}

\begin{defn}\cite[$(4.1)$]{CLN06}
 A metric $g$ for a manifold $M^n$ is called a gradient Ricci soliton if there is a smooth function $f: M^n \to \mathbb R$ such that
\begin{equation} \label{EA00}
Rc+\text{\text{Hess}} (f)+\frac{\lambda}{2} g=0.
\end{equation}
It is called steady when $\lambda=0$, shrinking when $\lambda=-1$ and expanding when $\lambda=1$.
\end{defn}

In \cite{H95} R. Hamilton proved the following identity for gradient steady Ricci solitons
\begin{equation}\label{EA02}
R+|\nabla f|^2=\Lambda
\end{equation}
where $\Lambda$ is a constant. Since on an ALE manifold the scalar curvature $R=O(r^{-2-\sigma})$, $|\nabla f|$ is bounded from \eqref{EA02}.
It can be proved, see for example in \cite[Theorem $4.1$]{CLN06}, that there exists an eternal solution $g(t) \, (-\infty < t < \infty)$ of the Ricci flow with $g(0)=g$ such that $g(t)=\phi(t)^*g$  where $\phi(t)$ is the $1$-parameter family of diffeomorphisms 
generated by $\nabla f$.

Since the solution $g(t)$ is self-similar, its curvature operater $|\text{Rm}(x,t)|$ is uniformly bounded as $|\text{Rm}(x,0)|$ is bounded for an ALE manifold. Moreover, $R \ge 0$ for every ancient complete solution of Ricci flow, see \cite[Corollary $2.5$]{BL07}. By the strong maximum principle either $R>0$ or $M$ is Ricci-flat. In the first case, it implies in particular that the constant $\Lambda$ in \eqref{EA02} is positive.

In addition, if the steady gradient Ricci soliton is nontrivial, the manifold has to be one-ended, see \cite[Corollary $1.1$]{MW11}.

Now we have

\begin{thm}\label{TA01}
If $(M^n,g)$ is an ALE manifold such that $g$ is a gradient steady Ricci soliton, then $g$ is Ricci-flat.
\end{thm}

\begin{proof}
(Nontrivial end) If $M^n$  is not Ricci-flat, we assume that \eqref{EA02} holds for a positive $\Lambda$. Moreover we assume that $|\Gamma|>1$.

From \eqref{EA02}, we have $|\nabla f| \le \Lambda^{\frac{1}{2}}$ and hence $f$ increases at most linearly. We can assume
\begin{align}\label{SA01}
|f(x)|\le C(1+r(x))
\end{align}
where $r$ is the function in the definition of ALE manifolds.

Now if we take any sequence $r_i \to +\infty$, then $(M, r_i^{-2}g)$ converges to $(\mathbb R^n /\Gamma,g_E)$ in the Gromov-Hausdorff sense. Moreover, the convergence is smooth away from $0$ by the Definition \ref{DA1}.
If we set $f_i=r_i^{-1}f$, then it is straightforward to see from \eqref{SA01} that $f_i$ are locally uniformly bounded on $\mathbb R^n/ \Gamma$.

Now by taking trace of \eqref{EA00}, we have 
\begin{equation}\label{EA03}
R+\Delta f=0.
\end{equation} 

Rewrite \eqref{EA03} in terms of $r_i^{-2}g$ and $f_i$, we have
\begin{align}\label{SA03}
\Delta_{g_i}f_i=r_i^2\Delta_gf_i=r_i\Delta_gf=-r_iR.
\end{align}

From the elliptic equation \eqref{SA03} and the fact that $R$ decays more than quadratically, $f_i$ converges to a function $f_E$ in $C_{\text{loc}}^{\infty}(\mathbb R^n/ \Gamma-\{0\})$. Moreover,
\begin{align}\label{SA04}
\Delta_{g_E}f_E=0,
\end{align}
By lifting everything from $\mathbb R^n/ \Gamma$ to $\mathbb R^n$, we know that since $f_E$ is a bounded harmonic function near $0$, it must be smooth on entire $\mathbb R^n$, see \cite[Theorem $3.9$]{ABR92}.

In addition,
\begin{align}\label{SA02}
|\nabla_{g_i}f_i|_{g_i}^2=r_i^2|\nabla_gf_i|^2_g=|\nabla_g f|_g^2=\Lambda-R.
\end{align} 
and hence by taking the limit we obtain
\begin{align}\label{SA05}
|\nabla_{g_E}f_E|^2=\Lambda.
\end{align}

Now from \eqref{EA00},
\begin{align}\label{SA06}
\left|\text{Hess}_{g_i}f_i\right|_{g_i}=r_i\left|\text{Hess}_{g}f_i\right|_g=r_i|Rc|_g.
\end{align}

Therefore, by taking the limit,
\begin{align}\label{SA07}
\left|\text{Hess}_{g_E}f_E\right|_{g_E}=0.
\end{align}

By considering \eqref{SA05} and \eqref{SA07}, we know that $f_E$ must be a nontrivial linear function. But it is not possible as $f_E$ is also defined on $\mathbb R^n/ \Gamma$.

(Trivial end):Assume that the ALE end $E$ of $M^n$ is trivial. 
From Theorem \ref{T201}, we can assume for all $\tau >0$, $\mu(g,\tau)<0$ since the Ricci flow solution of the steady soliton is eternal and $M^n$ is not Ricci-flat.

For any $\bar{\tau}>0$, by the monotonicity formula, $\mu(g(t),\bar{\tau}-t)$ is increasing for all $0 \le t < \bar{\tau}$. Therefore
$$
\mu(g(t),\bar{\tau}-t)=\mu(\phi(t)^*g,\bar{\tau}-t)=\mu(g,\bar{\tau}-t)
$$
is increasing for all $0 \le t < \bar{\tau}$. 
Since $\bar{\tau}$ can be any positive number, $\mu(g,\tau)$ is decreasing for all $\tau >0$. So it contradicts Theorem \ref{T202}.
Thus, the proof of Theorem \ref{TA01} is complete.
\end{proof}

For a complete Ricci shrinking soliton, we have
\begin{thm} \label{TA02}
If $(M^n,g)$ is an ALE manifold such that $g$ is a gradient shrinking Ricci soliton, then $(M^n,g)=(\mathbb R^n, g_E)$.
\end{thm}

It was proved in \cite{BLY11} that $\liminf_{d(x,O)\to \infty}R(x)d^2(x,O) >0$ for any non-flat shrinking soliton. So the proof of \ref{TA02} follows immediately since by the ALE condition $|\text{Rm}| \le Cr^{-2-\sigma}$.

There are nontrivial examples of expanding soliton on ALE manifolds, see the constructions in \cite{MTD03}.
\end{appendices}

\vspace{0.5in}

\vskip10pt

Yu Li, Mathematics Department, Stony Brook University,
Stony Brook, NY, 11794, USA;  yu.li.4@stonybrook.edu.\\

\end{document}